\documentclass[11pt]{article}

\usepackage{amssymb}
\usepackage{amsthm}
\usepackage{graphicx, gastex}
\usepackage{amsmath}
\usepackage{latexsym}
\usepackage{amsfonts}
\usepackage{color}

\DeclareSymbolFont{rsfscript}{OMS}{rsfs}{m}{n}
\DeclareSymbolFontAlphabet{\mathrsfs}{rsfscript}

\DeclareSymbolFont{rsfscript}{OMS}{rsfs}{m}{n}
\newtheorem{theorem}{Theorem}
\newtheorem{prop}{Proposition}
\newtheorem{defn}{Definition}
\newtheorem{lemma}{Lemma}
\newtheorem{cor}{Corollary}

\newtheorem{rem}{Remark}

\newcommand{\sch}{{Sch\"{u}tzenberger }}

\newcommand{\rf}{\rightarrow}

\newcommand{\la}{\langle}
\newcommand{\ra}{\rangle}

\begin{document}

\title{Maximal subgroups of amalgams of finite inverse semigroups.}

\author{ Alessandra Cherubini\\
   	Dipartimento di Matematica, Politecnico di Milano\\
 	Piazza L. da Vinci 32, I 20133 Milano, Italy\\
 	\texttt{alessandra.cherubini@polimi.it}\\
	\and
	Tatiana B. Jajcayov\' a \\
        Department of Applied Informatics, Comenius University\\
	Bratislava, Slovakia\\
        \texttt {tatiana.jajcayova@fmph.uniba.sk} \\
        \and
       Emanuele Rodaro \\
        Departemento de Matem´atica, Universidade do Porto\\
	Rua do Campo Alegre, 687, Porto, 4169-007, Portugal\\
 	\texttt{ emanuele.rodaro@fc.up.pt}}
\maketitle
\begin{abstract}

\end{abstract}

We use the description of the Sch\"utzenberger automata for amalgams of finite inverse semigroups given by Cherubini, Meakin, Piochi in \cite{Finite} to obtain structural results for such amalgams. \sch  automata, in the case of amalgams of finite inverse semigroups, are automata with special structure possessing finite subgraphs, that contain all essential information about the automaton. Using this crucial fact, and the Bass-Serre theory, we show that the maximal subgroups of an amalgamated free-product  are either isomorphic to certain subgroups of the original semigroups or can be described as fundamental groups of particular finite graphs of groups build from the maximal subgroups of the original semigroups.

\section{Introduction}

If $S_1$ and $S_2$ are semigroups (groups) such that $S_1 \cap S_2
= U$ is a non-empty subsemigroup (subgroup) of both $S_1$ and
$S_2$ then $[S_1 , S_2 ; U]$ is called an amalgam of semigroups
(groups). The amalgamated free-product $S_1 *_U S_2$ associated
with this amalgam in the category of semigroups (groups) is
defined by the usual universal diagram.

The amalgam $[S_1 , S_2 ; U]$ is said to be strongly embeddable in
a semigroup (group) $S$ if there are injective homomorphisms
$\phi_i : S_i  \rightarrow S$ such that $\phi_1 |_U = \phi_2 |U $
and $S_1 \phi_1 \cap S_2 \phi_2 = U\phi_1 = U\phi_2$. It is well
known that every amalgam of groups embeds in a group while
semigroup amalgams do not necessarily embed in any semigroup
\cite{Kim}. On the other hand, every amalgam of inverse
semigroups (in the category of inverse semigroups) embeds in an
inverse semigroup, and hence in the corresponding amalgamated free
product in the category of inverse semigroups \cite{Hall}.

An inverse semigroup is a semigroup $S$ with the property that for
each element $a \in S$ there is a unique element $a^{-1} \in S$
such that $a = a a^{-1} a$ and $a^{-1} = a^{-1} a a^{-1}$,
$a^{-1}$ is called the inverse of $a$. A consequence of the
definition is that the set of the idempotents $E(S)$ is a semilattice.
One may also define a natural partial order on
$S$ putting $a \leq b$ if and only if $a = e b$ for some $e\in E(S)$.

Inverse semigroups may be regarded as semigroups of partial
one-to-one transformations, so they arise very naturally in
several areas of mathematics and more recently also in computer
science,  mainly since the inverse of an
element can be seen as the ``undo with a trace"  of the
action represented by that element. We refer the reader to the
book of Petrich \cite{Petrich} for basic results and notation
about inverse semigroups and to the more recent books of Lawson
\cite{Lawson} and Paterson \cite{Paterson} for many references to
the connections between inverse semigroups and other branches of
mathematics.

The free object on a set $X$ in the category of inverse semigroup
is denoted by $FIS(X)$. It is the quotient of the free semigroup
$(X \cup X^{-1} )^+ $ by the least congruence $\nu$ that makes the
resulting quotient semigroup inverse (see \cite{Petrich}). The inverse
semigroup $S$ presented by a set $X$ of generators and a set $T$
of relations is denoted by $S = Inv\la X;T \ra$. This is the
quotient of the free semigroup $(X \cup X^{-1} )^+$ by the least
congruence $\tau$ that contains $\nu$ and the relations in $T$.

The structure of $FIS(X)$ was studied via graphical methods by
Munn \cite{munn}. Munn's work was greatly extended by Stephen
\cite{Steph} who introduced the notion of \sch graphs associated
with presentations of inverse semigroups. These graphs were widely
used in the study algorithmic problems and the structure of several
classes of inverse semigroups (see, for instance \cite{Ben2,FreeInverse,Finite,ChNuRoEquation,ChNuRo,ChRoAmVsHnn,HMM,Jajcayova,JMMS,RoByci,RoChe,RoSi,Steph98}). 
In particular Haataja, Margolis and Meakin were the
first to show that Bass-Serre theory may be applied to study the
structure of maximal subgroups, to obtain results for amalgams of inverse
semigroups where $U$ contained all idempotents of $S_1$ and $S_2$.
Their construction was extended by Bennett \cite{Ben2} and
Jajcayov\'a \cite{Jajcayova} to respectively study the maximal subgroups of
a special class of amalgams and HNN-extensions of inverse
semigroups.

In \cite{Finite},  the word problem for amalgams of
finite inverse semigroups was shown to be decidable by constructing an automaton that is a good approximation of the \sch
automaton. Here, we make use of this construction to study
the structure of maximal subgroups in such amalgams. 
This  is done along the lines of Bennett's study of
maximal subgroups of lower bounded amalgams, but as
amalgams of finite inverse semigroups are not necessarily lower
bounded, \sch automata of amalgams of finite inverse semigroups
differ from  Bennett's automata, mainly by the fact the
\sch graphs of the two original semigroups do not appear as subgraphs of the resulting \sch automaton of the amalgam. This leads to several important technical differences in the treatment and in the results. 

The paper is organized as follows: in Section \ref{sec:preliminaires}, we recall basic definitions and relevant results concerning   \sch automata of inverse semigroups, and the structure  and properties of \sch automata of amalgams of finite inverse semigroups in particular. In Section \ref{sec: automorphism} we prove that the automorphism groups of the \sch graphs of  amalgams (isomorphic to the maximal subgroups) are isomorphic to the automorphism groups of particular subgraphs with special properties.  We study these subgraphs and their properties in Section \ref{sec: union of hosts}. In Section \ref{Sect.BS} we give a brief review of the Bass-Serre theory of groups acting on graphs. Finally, merging this theory with the previous results, in Section \ref{sec: maximal sub}, we give a complete description of the maximal subgroups in an amalgam of finite inverse semigroups.

\section{Preliminaries}\label{sec:preliminaires}

In this section we review  definitions and  results concerning   \sch automata of inverse semigroups, and briefly describe the
construction of \sch graphs of amalgams of finite inverse
semigroups. We
refer the reader to \cite{Ben,Finite,Petrich,Steph} for more details.

An \emph{inverse word graph} over an alphabet $X$ is a strongly
connected labelled digraph whose edges are labelled over $X \cup
X^{-1}$, where $X^{-1}$ is the set of formal inverses of elements
in $X$, so that for each edge $e$ labelled by $x\in X$ there is an
edge labelled by $x^{-1}$ in the reverse direction. A finite
sequence of edges $e_i=(\alpha_i,a_i,\beta_i),\ 1\leq i\leq n,
a_i\in X\cup X^{-1}$ with $\beta_i=\alpha_{i+1}$ for all $i$ with
$1\leq i<n$, is an $\alpha_1-\beta_n$ path of $\Gamma$ labelled by
$a_1a_2\ldots a_n\in (X\cup X^{-1})^+$. An {\it inverse automaton}
over $X$ is a triple $\mathcal{A}=(\alpha, \Gamma, \beta)$ where
$\Gamma$ is an inverse word graph over $X$ with set of vertices $V(\Gamma)$ and $\alpha, \beta\in
V(\Gamma)$ are two special vertices called the initial and final state of
$\mathcal{A}$. The language $L[\mathcal{A}]$ recognized by
$\mathcal{A}$ is the set of labels
of all $\alpha-\beta$ paths of $\Gamma$.\\
The inverse word graph $\Gamma$ over $X$ is \emph{deterministic}
if for each $\nu\in V(\Gamma)$, $a\in X\cup X^{-1}$,
$(\nu,a,\nu_1),(\nu,a,\nu_2)\in Ed(\Gamma)$ implies $\nu_1=\nu_2$.

Morphisms between inverse word graphs are graph morphisms that
preserve labelling of edges and are referred to as {\it
$V$-homomorphisms} in \cite{Steph}. In this paper we simply refer to them as homomorphisms, and in the case a morphism is surjective, injective or bijective, we refer to it as an epimorphism, monomorphism or isomorphism, respectively. The group of  all automorphisms of graph $\Gamma$ is denoted by $Aut(\Gamma)$. If $\Gamma$ is an
inverse word graph over $X$ and $\rho$ is an equivalence relation
on the set of vertices of $\Gamma$, the corresponding quotient
graph $\Gamma /\rho$ is called a \emph{$V$-quotient} of $\Gamma$
(see \cite{Steph} for details). There is a least equivalence
relation on the vertices of an inverse automaton $\Gamma$ such
that the corresponding $V$-quotient is deterministic. A
deterministic $V$-quotient of $\Gamma$ is called a
\emph{$DV$-quotient}. There is a natural homomorphism from $\Gamma$ onto a $V$-quotient of $\Gamma$. The notions of morphism, $V$-quotient and $DV$-quotient of
inverse graphs extend  analogously to  inverse
automata. (see \cite{Steph}).

Let $S = Inv \la X ; T \ra\simeq (X\cup X^{-1})^+/\tau$ be an inverse
semigroup. The \sch graph $S\Gamma (X,T;w)$ for a word $w \in (X
\cup X^{-1} )^+ $ relative to the presentation $\langle X|T
\rangle$ has the $\cal R$-class of $w\tau$ in $S$ as the set of
vertices and its edges consist of all the triples $(s,x,t)$ with $x \in X \cup
X^{-1}$, and $s\cdot x\tau = t$. We view this edge as being
directed from $s$ to $t$. The graph $S\Gamma (X,T;w)$ is a
deterministic inverse word graph over $X$. The structure of \sch
graphs is strictly connected with the Green's relations on $S$. In
particular the following results by Stephen will be important for
our purposes.
\begin{prop}\label{prop:old Ste}
Let $S= Inv \la X ; T \ra$ be an inverse semigroup and let $e,f\in
E(S)$. Then
\begin{enumerate}
    \item $e \mathcal{D} f$ if and only if there exists a
    $V$-isomorphism $\phi: S\Gamma (X,T;e)\rightarrow S\Gamma
    (X,T;f)$ \cite[Theorem 3.4 (a)]{Steph}.
    \item The $\mathcal{H}$-class of $e$ and $Aut(S\Gamma (X,T;e))$
    are isomorphic groups \cite[Theorem 3.5]{Steph}.
\end{enumerate}
\end{prop}
The second statement identifies the group of symmetries of $S\Gamma (X,T;e)$ with the maximal subgroup of $S$ having $e$ as unity. This fact is fundamental since it gives us a geometric interpretation of these maximal subgroups and it will be implicitly used throughout the paper. The automaton ${\cal A} (X,T;w)$ whose underlying graph is $S\Gamma (X,T;w)$ with the vertex $ww^{-1}\tau$ as the initial state and the vertex $w\tau$ as the terminal state, is called the \sch automaton
of $w \in (X \cup X^{-1} )^+ $ relative to the presentation
$\langle X|T \rangle$.

In \cite{Steph} Stephen provides an iterative (but in general not
effective) procedure to build ${\cal A} (X,T;w)$ via two operations,
\emph{the elementary determination} and \emph{the elementary
expansion}. We briefly recall such operations. Let $\Gamma$ be an inverse word graph over $Y$, an
elementary determination consists of folding two edges starting from
the same vertex and labeled by the same letter of the alphabet
$Y\cup Y^{-1}$. The elementary expansion applied to $\Gamma$
relative to a presentation $\langle Y|T\rangle$ consists in adding a
path $(\nu_1,r,\nu_2)$ to $\Gamma$ wherever $(\nu_1,t,\nu_2)$ is a
path in $\Gamma$ and $(r,t)\in T\cup T^{-1}$. An inverse word graph
is called \emph{closed} with respect to $\langle Y|T\rangle$ if it
is a deterministic word graph where no expansion relative to
$\langle Y|T\rangle$ can be performed. An inverse automaton is
\emph{closed} with respect to $\langle Y|T\rangle$ if its underlying
graph is so.

Let $S_i = Inv \la X_i ; R_i \ra = (X_{i} \cup X_{i}^{-1})^{+} /
\eta_{i}$, $i=1,2$, where the $X_i$ are disjoint alphabets. For an
amalgam $[S_1 , S_2 , U]$  we view the natural image of $u\in U$ in
$S_i$ under the embedding of $U$ as a word in the alphabet $X_i$ and
$\la X_{1} \cup X_{2}| R_1 \cup R_2 \cup W\ra$ with
$W=\{(\phi_1(u) , \phi_2(u))|u \in U\}$ is a presentation of
 $S_1*_US_2$. We put $X=X_1\cup X_2$ and $R=R_1\cup R_2$ and
we call $\langle X|R\cup W\rangle$ the \emph{standard presentation
of $S_1 *_US_2$ with respect to the presentations of $S_1$ and
$S_2$}, for short the standard presentation of $S_1*_US_2\simeq
(X\cup X^{-1})^+/\tau$. In the sequel we will use the superscript notations $\mathcal{D}^{U},\mathcal{D}^{S_i}$, $\mathcal{D}^{S_1 *_US_2}$ to discriminate the $\mathcal{D}$-classes in $U,S_i, (S_1 *_US_2)$, respectively. We keep this convention for all the Green's relations as well as for their classes. For instance, for the maximal subgroup in $S_{1}$ of an idempotent $e$ in $S_{1}$ we use the symbol $H_{e}^{S_{1}}$, instead if we consider the maximal subgroup in the free-product with amalgamation we use the notation $H_{e}^{S_{1*_{U}S_{2}}}$.
We adhere to this notation throughout the remain of the paper and we always assume that $S_1$ and $S_2$ are finite inverse semigroups.

Let $\Gamma$ be an inverse word graph labeled over $X = X_{1} \cup
X_{2}$ with $X_{1} \cap X_{2}=\emptyset$, an edge of $\Gamma$
that is labeled from $X_{i} \cup X_{i}^{-1}$ (for some $i \in
\{1,2\}$) is said to be \emph{colored} $i$. A subgraph of $\Gamma$
is called \emph{monochromatic} if all its edges have the same
color. A \emph{lobe} of $\Gamma$ is defined to be a maximal
monochromatic connected subgraph of $\Gamma$. The coloring of
edges extends to coloring of lobes. Two lobes are said to be
\emph{adjacent} if they share common vertices, called
\emph{intersection vertices}. If $\nu \in V(\Gamma)$ is an
intersection vertex, then it is common to two unique lobes, which
we denote by $\Delta_{1}(\nu)$ and $\Delta_{2}(\nu)$, colored $1$ and $2$, respectively.
We define the \emph{lobe graph} of
$\Gamma$ to be the graph whose vertices are the lobes of $\Gamma$
and whose edges correspond to adjacency of
lobes.\\
We remark that a nontrivial inverse word graph $\Delta$ colored
$i$ and closed relative to $\la X_i|R_i\ra$ contains all the paths
$(\nu_1,v',\nu_2)$ with $v'\in (X_i\cup X_i^{-1})^+$ such that
$v'\eta_i=v\eta_i$, provided that $(\nu_1,v,\nu_2)$ is a path of
$\Delta$. Hence we often say that there is a path
$(\nu_1,s,\nu_2)$ with $s\in S_i$ in $\Delta$ whenever
$\{(\nu_1,v,\nu_2)| v\eta_i=s\}\neq\emptyset$. Similarly we say
that $(\nu_1,u,\nu_2)$ with $u\in U$ is a path of $\Delta$ to mean
that $(\nu_1,\phi_i(u),\nu_2)$ is a path of $\Delta$. For all
$\nu\in V(\Delta)$ we denote by $\mathcal{L}_U(\nu,\Delta)$ the
set of all the elements $u\in U$ such that $(\nu,u,\nu)$ is a loop
based at $\nu$ in $\Delta$. Thus following the notation of
\cite{Finite}, that slightly modifies the analogous one introduced
in \cite{Ben}, we say that an inverse word graph $\Gamma$ on $X$
is an \emph{opuntoid graph} if it satisfies the following
properties.
\begin{itemize}
    \item $\Gamma$ is a deterministic inverse word graph;
    \item the lobe graph  is a tree, denoted $T(\Gamma)$;
    \item each lobe of $\Gamma$ is a finite closed DV-quotient of a
    Sch\"{u}tzenberger graph relative to $\langle X_i|R_i\rangle$, $i\in\{1,2\}$;
    \item (\emph{loop equality property}) for each intersection vertex $\nu$
    of $\Gamma$,
    $\mathcal{L}_U(\nu,\Delta_1(\nu))=\mathcal{L}_U(\nu,\Delta_2(\nu))$;
    \item (\emph{assimilation property}) for each intersection vertex $\nu$
    of $\Gamma$, and for each $\nu'\in V(\Delta_1(\nu))\cap V(\Delta_2(\nu))$ with $\nu'\neq \nu$
    there is $u\in U$ such that
    $(\nu,u,\nu')$ is a path in both $\Delta_1(\nu)$ and $\Delta_2(\nu)$.
    Moreover $(\nu,u',\nu')$ with $u'\in U$ is path in $\Delta_1(\nu)$
    if and only if $(\nu,u',\nu')$ is a path in $\Delta_2(\nu)$.
\end{itemize}
We remark that assimilation property implies that if $\nu$ is an
intersection vertex of $\Gamma$ and $(\nu,u',\nu')$ is a path of
$\Gamma$, then $\Delta_i(\nu)=\Delta_i(\nu')$ with $i\in \{1,2\}$.
This property is referred as the \emph{related pair separation property}
in \cite{Ben}. A \emph{subopuntoid subgraph} $\Theta$ is an inverse word
subgraph of an opuntoid graph $\Gamma$ such that if $\Delta$ is a
lobe of $\Theta$, then $\Delta$ is also a lobe of $\Gamma$. If $\phi:\Gamma\rf\Gamma'$ is a homomorphism between the two opuntoid graphs $\Gamma,\Gamma'$ and $\Theta$ is a subopuntoid subgraph of $\Gamma$, we often denote $\phi|_{\Theta}$ the restriction of $\phi$ to $\Theta$. Opuntoid
automata are automata whose underlying graphs are so. 
Since $S_1, S_2$ are finite, for each lobe $\Delta$ of an opuntoid
graph $\Gamma$ colored $i\in\{1,2\}$ and for each $\nu\in
V(\Delta)$, there is a minimum idempotent, denoted by $e_i(\nu)$,
labeling a loop based at $\nu$. Moreover if
$\mathcal{L}_U(\nu,\Delta)\neq\emptyset$ then there is also a
minimum idempotent belonging to $\mathcal{L}_U(\nu,\Delta)$ which
is denoted by $f(e_i(\nu))$.

In \cite{Finite} the \sch automaton $\mathcal{A}(X,R\cup W;w)$ is
built by means of a sequence of constructions. The first three
constructions are iteratively applied finitely many times to the
linear automata $(\alpha, lin(w),\beta)$. These constructions are
applied on the lobes and the intersection vertices, they do not
increase the number of lobes of $lin(w)$ and after applying the
fourth one, a finite opuntoid automaton $Core(w)=(\alpha,
\Gamma_0, \beta)$ called \emph{Core automaton} or briefly
\emph{Core} of $w$ is obtained. $Core(w)$ contains all the
information to build the \sch automaton of $w$ and results to be a
subopuntoid automaton of the \sch automaton of $w$. It is closed
with respect to $\langle X|R \rangle$ but it is not in general
closed relative to $\langle X|R\cup W \rangle$ and approximates
$\mathcal{A}(X,R\cup W;w)$ in the sense of Stephen \cite{Steph}.

Then a last construction, called Construction 5, is applied in
general infinitely many times to $Core(w)$ to obtain
$\mathcal{A}(X,R\cup W;w)$. Let $\Gamma$ be an opuntoid graph and
let $\Delta$ be a lobe colored $i\in\{1,2\}$. Then $\nu\in
V(\Gamma)$ is a \emph{bud} of $\Gamma$ (see \cite{Ben}) if it is
not an intersection vertex and
$\mathcal{L}_U(\nu,\Delta)\neq\emptyset$. The graph $\Gamma$ is
\emph{complete} if it has no bud.
\\
\\
\textbf{Construction 5}(see \cite{Finite})
\begin{itemize}
\item [-]Let $\Gamma$ be a non-complete opuntoid graph and let $\nu\in
V(\Delta)$ be a bud belonging to a lobe $\Delta$ colored by some
$i\in \{1,2\}$. Then $\mathcal{L}_U(\nu,\Delta)\neq\emptyset$. Let
$f=f(e_i(\nu))$ and let
$(x,\Lambda,x)=\mathcal{A}(X_{3-i},R_{3-i};f)$. Consider the set of vertices (called as \emph{net}):
$$
N(x,\Lambda)=\{y\in V(\Lambda):(x,u,y)\:\mbox{is a path
in}\;\Lambda\;\mbox{for some}\; u\in\mathcal{L}_U(\nu,\Delta)\}
$$
and let $\rho\subseteq V(\Lambda)\times V(\Lambda)$ be the least
equivalence relation that identifies each vertex of $N(x,\Lambda)$
with $x$ and such that $\Lambda/\rho$ is deterministic and put
$\Delta'=\Lambda/\rho$. Consider the inverse word automaton
$\mathcal{B}=(\nu,\Gamma,\nu)\times(x\rho,\Delta',x\rho)$ (see
\cite{Steph}), then for all $u\in U$ such that $(\nu,u,y)$ and
$(\nu,u,y')$ are paths of $\Delta$ and $\Delta'$ respectively,
consider the equivalence relation $\kappa$ on $V(\mathcal{B})$
that identifies $y$ and $y'$ and call $\overline{\Gamma}$ the
underlying graph of $\mathcal{B}/\kappa$.
\end{itemize}
Lemma 10 of \cite{Finite} states that $\overline{\Gamma}$ is an
opuntoid graph with one more lobe than $\Gamma$ and the graph
$\Gamma$ is unchanged by this construction. 

Let $\Delta$ and $\Delta'$ be two adjacent lobes of an opuntoid
graph $\Gamma$. Following \cite{Ben,Rthesys}, we say that
$\Delta'$ \emph{directly feeds off} $\Delta$, in symbols
$\Delta\mapsto\Delta'$, if $\Delta'$ can be obtained from $\Delta$
by applying Construction 5 at some intersection vertex $\nu\in
V(\Delta)\cap V(\Delta')$. Moreover for each pair of lobes
$\Delta,\Delta'$ of $\Gamma$ we say that $\Delta'$ feeds off
$\Delta$, $\Delta\mapsto^*\Delta'$, if $\Delta$ and $\Delta'$ are
related in the transitive closure of $\mapsto$. We remark that the
definition of feeding off in our case differs from the one given
in \cite{Ben} because in Bennett's case no quotient of the lobe
$\Lambda$ is needed.

For each finite opuntoid graph $\Gamma$ consider the sequence
$$
\Gamma=\Gamma_1\hookrightarrow\Gamma_2\hookrightarrow\ldots\Gamma_m\hookrightarrow\ldots
$$
where $\Gamma_{m+1}$ is obtained from $\Gamma_m$ applying
Construction 5 at some bud of $V(\Gamma_m)$ and consider the directed limit $\varinjlim \Gamma_i
= \bigcup_{k>0} \Gamma_k$. Then $\varinjlim \Gamma_i$ is a closed
automaton with respect to $\langle X|R\cup W\rangle$ (see \cite{Finite}); this operation is called in the sequel the closure of $\Gamma$ and it is denoted by
$cl_{R\cup W}(\Gamma)$. The uniqueness of the closure of an
opuntoid graph follows from the work of Stephen
\cite{Steph,Steph98}. Note that the closure of the
underlying graph $\Gamma_0$ of $Core(w)$ is the \sch graph
$S\Gamma(X,R\cup W;w)$. This graph is in general an
infinite graph, however a direct consequence of the finiteness of
$S_1, S_2$ is that there are finitely many different types of
lobes up to isomorphisms.

We say that a lobe $\Delta'$ of an opuntoid graph $\Gamma$ that is adjacent to
precisely one other lobe $\Delta$ of $\Gamma$ is called a
\emph{parasite} (see also \cite{Ben2}) of $\Gamma$ if $\Delta'$ feeds
off $\Delta$.

A subopuntoid subgraph $\Gamma'$ of an opuntoid graph $\Gamma$ is
called a \emph{host} of $\Gamma$ if:
\begin{itemize}
    \item its lobe tree is finite,
    \item it is parasite-free,
    \item every lobe of $\Gamma$ not belonging to $\Gamma'$ feeds off some lobe of $\Gamma'$.
\end{itemize}
A host of an opuntoid automaton is a host of its underlying graph.

It is straightforward to see that a host $\Theta$ of an opuntoid
graph $\Gamma$ is a minimal subopuntoid subgraph of $\Gamma$ such
that $cl_{R\cup W}(\Theta)\supseteq\Gamma$. Moreover we have the
following proposition whose statement and proof are formally equal
to the ones of \cite[Lemma 6.2]{Ben} even if the definitions
differ in some technical details

\begin{prop}\label{Prop:host set}
Let $\Gamma$ be an opuntoid graph. Then a host of $\Gamma$ is a
maximal parasite-free subopuntoid subgraph of $\Gamma$. If $\Gamma$
has more than one host, then every host is a lobe and the unique
reduced lobe path connecting any two hosts consists entirely of
lobes that are hosts.
\end{prop}

Obviously an opuntoid with finite-lobe graph has always a host, and since
the underlying graph $\Gamma_0$ of $Core(w)$ has finitely many
lobes and $S\Gamma(X,R;w)=cl_{R\cup W}(\Gamma_0)$, by Lemma 6.1 of
\cite{Ben} we derive that the \sch graph $S\Gamma(X,R;w)$ posses
always a host contained in $\Gamma_0$.\\
Moreover the union of all hosts of an opuntoid graph $\Gamma$ that
has a host is a subopuntoid subgraph of $\Gamma$, denoted by
$Host(\Gamma)$.

The closure of an opuntoid automaton $(\alpha,\Gamma,\beta)$ is
$(\alpha', cl_{R\cup W}(\Gamma),\beta')$ where $\alpha', \beta'$ are
the natural images of $\alpha$ and $\beta$ respectively in
$cl_{R\cup W}(\Gamma)$. Then we have the following description of
\sch automata that slightly extends Theorem 3 in
\cite{Finite} (see also \cite{Rthesys})

\begin{prop}\label{Prop:Schutzenberger automata}
Let $S = S_1*_U S_2$ be an amalgamated free-product of finite
inverse semigroups $S_1$ and $S_2$ amalgamating a common inverse
subsemigroup $U$, where $\la X_i | R_i\ra$ are presentations of
$S_i$ for $i = 1, 2$. Let $X = X_1\cup X_2$, $R = R_1\cup R_2$ and
$W$ be the set of all pairs $(\phi_1(u),\phi_2(u))$ for $u \in
U$. Then the \sch automata relative to $\langle X | R \cup
W\rangle$ are complete opuntoid automata possessing a host.
\end{prop}

It is important for the sequel to point out the differences
between the definitions of opuntoid graphs, feeding off relation
and \sch automata given in Bennett's papers \cite{Ben,Ben2}
and the ones presented here. According to Bennett, lobes of
opuntoid graphs are \sch graphs relative to $\langle
X_i|R_i\rangle$ for some $i\in\{1,2\}$, instead here lobes are in general
$DV$-quotients of such graphs. The lower bound equality property of
opuntoid graphs according to Bennett is here replaced by the loop
equality property that, as remarked in \cite{Finite} p.10,
coincides with lower bound equality property in the case that
lobes are \sch graphs. In \cite{Ben} a lobe $\Delta'$ colored $i$
directly feeds off $\Delta$ in a vertex $\nu$ if
$\Delta'=S\Gamma(X_i,R_i;f(e_{3-i}(\nu)))$ while in our case some
$DV$-quotient is in general needed to guarantee the loop equality
property in $\nu$. Lastly, \sch automata of lower bounded amalgams
are completely characterized as complete opuntoid automata with
hosts, \sch automata of amalgams of finite semigroups are only
described in such a way because complete opuntoid automata with
hosts are not necessarily \sch automata.

\section{Automorphisms of opuntoid graphs}\label{sec: automorphism}

We shall adopt the same notion as in  the previous section for the presentations of $S_1$, $S_2$ and $S_1*_US_2$ and all the considered opuntoid graphs shall be
assumed to be determined by these presentations. The main result proved here, shows that there is a group isomorphism between the automorphism group
of a complete opuntoid graph $\Gamma$ and the automorphism group
of the subopuntoid $Host(\Gamma)$ formed by the union of all the hosts of $\Gamma$.
\\
Also in our case we can state Lemma 6 and Lemma 7 of \cite{Ben2},
namely the proofs of these lemmas in \cite{Ben2} only use the
facts that the lobe graphs are trees, homomorphisms of graphs
preserve labels (and so colors) of edges and automorphisms of
deterministic inverse word graphs that agree on a vertex are
equal.
\begin{lemma}\label{lemma:InducedAuthomorphismOnTree} Let $\Gamma,\
\Gamma'$ be opuntoid graphs and let
$\varphi:\Gamma\rightarrow\Gamma'$ be a homomorphism. Then $\varphi$
is an isomorphism if and only if it induces an automorphism of lobe
trees and maps lobes isomorphically onto lobes.
\end{lemma}
Let $\Gamma$ be an opuntoid automaton with finitely many lobes. It is straightforward to check
that the automorphism group of $\Gamma$ embeds in the
automorphism group of some lobe of $\Gamma$ by using the fact that the automorphism group of a finite tree fixes a vertex or an edge (see \cite[Subsection 27.1.3]{Bab}).

Let $\Gamma$ be an opuntoid graph and let $\Theta$ be a subopuntoid
subgraph of $\Gamma$. For any lobe $\Delta$ not belonging to
$\Theta$ the notation $\Theta\cup\Delta$ will denote the least
subopuntoid subgraph of $\Gamma$ which contains $\Theta$ and
$\Delta$. We have the following lemma.
\begin{lemma}\label{lemma:ExtensionProperty}
Let $\Gamma,\Gamma'$ be two complete opuntoid graphs and let
$\Theta, \Theta'$  be subopuntoid subgraphs containing a host of
$\Gamma$ and $\Gamma'$ respectively. Let $\nu$ be a bud of
$\Theta$ and let $\Delta_i(\nu)$ be a lobe of $\Theta$ for some
$i\in\{1,2\}$. Let $\varphi:\Theta\rightarrow\Theta'$ be an
isomorphism and let $\nu'=\varphi(\nu)$. Then $\varphi$ can be
extended to an isomorphism from $\Theta\cup\Delta_{3-i}(\nu)$ onto
$\Theta'\cup\Delta_{3-i}(\nu')$.
\end{lemma}
\begin{proof}
Since $\nu$ is a bud of $\Theta$ and $\Delta_i(\nu)$ is a lobe of
$\Theta$, the lobe $\Delta_{3-i}(\nu)$ of $\Gamma$ does not belong
to $\Theta$. Moreover since $\Theta$ contains a host of $\Gamma$,
then $\Delta_i(\nu)\mapsto\Delta_{3-i}(\nu)$. Let $f=f(e_i(\nu))$
be the minimum idempotent in $U$ labelling a loop based at $\nu$
in $\Delta_i(\nu)$. Since $\phi$ preserves the labels, then
$f=f(e_i(\varphi(\nu)))$. Since $\Gamma'$ is complete,
$\varphi(\nu)$ is not a bud, hence it is an intersection vertex.
Moreover, by Lemma \ref{lemma:InducedAuthomorphismOnTree},
$\Delta_i(\varphi(\nu))$ is a lobe of $\Theta'$.  If also
$\Delta_{3-i}(\varphi(\nu))$ is a lobe of $\Theta'$ then again by
Lemma \ref{lemma:InducedAuthomorphismOnTree}
$\Delta_i(\nu)=\varphi^{-1}(\Delta_i(\varphi(\nu))$ and
$\varphi^{-1}(\Delta_{3-i}(\varphi(\nu)))$ are adjacent lobes in
$\Theta$ with intersection vertex $\nu$, hence $\nu$ is not a bud
of $\Theta$, a contradiction. So $\Delta_{3-i}(\varphi(\nu))$ is
not a lobe of $\Theta'$ and
$\Delta_{i}(\varphi(\nu))\mapsto\Delta_{3-i}(\varphi(\nu))$
because $\Theta'$ contains a host. Since
$f(e_i(\varphi(\nu)))=f(e_i(\nu))=f$ and
$\mathcal{L}_U(\nu,\Delta_i(\nu))=\mathcal{L}_U(\varphi(\nu),\Delta_i(\varphi(\nu)))$,
by definition of direct feed off, $(\nu,\Delta_{3-i}(\nu),\nu)$
and $(\varphi(\nu),\Delta_{3-i}(\varphi(\nu)),\varphi(\nu))$ are
isomorphic to the same $DV$-quotient of the \sch automaton
$\mathcal{A}(X_{3-i},R_{3-i};f)$ and so the lobes
$\Delta_{3-i}(\nu)$ and $\Delta_{3-i}(\varphi(\nu))$ are
isomorphic under an isomorphism $\psi$ such that
$\psi(\nu)=\varphi(\nu)$. So it is straightforward to see that $\varphi$ can be extended to an
isomorphism from $\Theta\cup\Delta_{3-i}(\nu)$ onto
$\Theta'\cup\Delta_{3-i}(\varphi(\nu))$ whose restriction to
$\Delta_{3-i}(\nu)$ is $\psi$.
\end{proof}

From this lemma we deduce the following property.

\begin{prop}\label{Prop:ExtensionToIsomorphism}
Let $\Gamma,\Gamma'$ be two complete opuntoid graphs and let
$\Theta, \Theta'$ be subopuntoid subgraphs containing a host of
$\Gamma,\Gamma'$ respectively. Let $\varphi:\Theta\rf\Theta'$ be
an isomorphism. Then $\varphi$ can be extended to an isomorphism
$\varphi^*:\Gamma\rf\Gamma'$.
\end{prop}
\begin{proof}
Let $\mathcal{P}$ be the set of the pairs
$(\phi,\overline{\Gamma})$ where $\overline{\Gamma}$ is a
subopuntoid subgraph of $\Gamma$ containing $\Theta$ and $\phi$ is
a graph monomorphism of $\overline{\Gamma}$ into $\Gamma'$ such that
$\phi|_{\Theta}=\varphi$. Obviously $(\varphi,\Theta)\in
\mathcal{P}$. Let $\le$ be the natural partial order on
$\mathcal{P}$ defined by $(\phi_{1},\Gamma_{1})\le (\phi_{2},\Gamma_{2})$ if $\Gamma_{1}$ is a subopuntoid subgraph of $\Gamma_{2}$ and
$\phi_{2}|_{\Gamma_{1}}=\phi_{1}$. By Hausdorff maximality lemma there is a maximal chain
$\Omega=\{(\phi_{\alpha},\Gamma_{\alpha})\}_{\alpha\in I}$. Consider the pair
$(\widehat{\phi},\widehat{\Gamma})$ where
$\widehat{\Gamma}=\cup_{\alpha}\Gamma_{\alpha}$ and
$\widehat{\phi}(\nu)=\phi_{\alpha}(\nu)$ for $\nu\in
V(\Gamma_{\alpha})$. It is easy to show that the element
$(\widehat{\phi},\widehat{\Gamma})$ belongs to $\mathcal{P}$ and
in particular it is a maximal element of the chain $\Omega$.
\\
We claim $\widehat{\Gamma}=\Gamma$. Indeed suppose that, contrary to our claim, $\widehat{\Gamma}\neq\Gamma$.
Then $\widehat{\Gamma}$ has a
bud $\nu$ and so only one of the two lobes
$\Delta_1(\nu),\Delta_2(\nu)$ of $\Gamma$ is in
$\widehat{\Gamma}$. Assume without loss of generality that it is
$\Delta_1(\nu)$. Then by Lemma \ref{lemma:ExtensionProperty} the
monomorphism $\widehat{\phi}$ can be extended to an isomorphism
from $\widehat{\Gamma}\cup \Delta_2(\nu)$ onto a subopuntoid graph
of $\Gamma'$ against the maximality of
$(\widehat{\phi},\widehat{\Gamma})$ whence
$\widehat{\Gamma}=\Gamma$ and $\widehat{\phi}$ is an isomorphism
between the opuntoid graph $\Gamma$ onto the subopuntoid subgraph
$\widehat{\phi}(\Gamma)\subseteq \Gamma'$. Suppose
$\widehat{\phi}(\Gamma) \neq \Gamma'$, then
$\widehat{\phi}(\Gamma)$ has a bud $\mu$ and again only one of the
two lobes $\Delta_1(\mu),\Delta_2(\mu)$ of $\Gamma'$ belongs to
$\widehat{\phi}(\Gamma)$. Repeating the above argument on the
isomorphism
$\widehat{\varphi}^{-1}:\widehat{\phi}(\Gamma)\rightarrow \Gamma$,
we get $\widehat{\phi}(\Gamma)=\Gamma'$.
\end{proof}

We are now in position to prove the following proposition.

\begin{prop}\label{prop:IsomorphismOpuntoidHost}
Let $\Gamma$ be a complete opuntoid graph which posses a host.
Then the automorphism group of $\Gamma$ is isomorphic to the
automorphism group of the union of all hosts of $\Gamma$.
\end{prop}
\begin{proof}
Let $Host(\Gamma)$ be the union of all hosts of $\Gamma$. We
know from Proposition \ref{Prop:ExtensionToIsomorphism}, that each
$\phi\in Aut(Host(\Gamma))$ can be extended to an
automorphism $\varphi\in Aut(\Gamma)$. We prove that $\varphi$
preserves the feed off relation. Assume that
$\Delta\mapsto\Delta'$ and let $\nu\in V(\Delta)\cap V(\Delta')$.
Then obviously $\varphi(\nu)\in V(\varphi(\Delta))\cap
V(\varphi(\Delta'))$, moreover, if $i\in \{1,2\}$ is the color of
$\Delta$, then $f(e_i(\nu))= f(e_i(\varphi(\nu)))$ and
$\mathcal{L}_U(\nu,\Delta)=\mathcal{L}_U(\varphi(\nu),\varphi(\Delta))$.
Let $f=f(e_i(\nu))$, since $\Delta\mapsto\Delta'$ then
$$
(\nu,\Delta',\nu)\simeq \mathcal{A}(X_{3-i},R_{3-i};f)/\rho
$$
where $\rho$ is the least equivalence relation that identifies the initial
vertex $\nu$ of $\mathcal{A}(X_{3-i},R_{3-i};f(e_i(\nu)))$ with the
final vertices $y$ of all the paths $(\nu,u,y)$ with $u\in
\mathcal{L}_U(\nu,\Delta)$ and makes the quotient deterministic. Thus 
$$
(\varphi(\nu),\varphi(\Delta'),\varphi(\nu))\simeq
\mathcal{A}(X_{3-i},R_{3-i};f)/\rho
$$
and so by the definition of direct feed off, $\varphi(\Delta)\mapsto\varphi(\Delta')$. Therefore $\varphi$ sends hosts into hosts and so $\varphi|_{Host(\Gamma)}$
belongs to $Aut(Host(\Gamma))$. It is straightforward to
check that the map $\chi$ defined by
$\chi(\varphi)=\varphi|_{Host(\Gamma)}$ is a group
isomorphism from $Aut(\Gamma)$ onto $Aut(Host(\Gamma))$.
\end{proof}

\section{Union of hosts of \sch graphs}\label{sec: union of hosts}

In this section we consider the union of hosts of a \sch graph of
the free amalgamated product $S_1*_U S_2$ of the finite inverse
semigroups $S_1$,$S_2$. First we characterize the \sch graphs
of the free amalgamated product $S_1*_U S_2$ of the finite inverse
semigroups $S_1$,$S_2$ with more than one host. We need to recall
some results from \cite{RoByci}.

\begin{prop}\label{prop:hosts which are sch} \cite[Proposition
10]{RoByci}
\\
Let $\Gamma$ be an opuntoid graph. Let $\Delta, \Delta'$ be
two lobes of $\Gamma$ colored respectively by $i, 3-i$ for some $i
= 1, 2$ with $\Delta\mapsto\Delta'$. Let $\nu\in V(\Delta)\cap
V(\Delta')$ be an intersection vertex of $\Gamma$. Then $f =
f(e_i(\nu)) = e_{3-i}(\nu)\in E(U)$. Conversely if $\Delta
\simeq S\Gamma(X_i,R_i; f)$ is a lobe  of $\Gamma$ and $f\in E(U)$,
then $\Delta' \simeq S\Gamma(X_{3-i},R_{3-i}; f)$ is a lobe of
$\Gamma$ and $\Delta'\mapsto\Delta$.
\end{prop}

\begin{prop}\label{prop:lifting} \cite[Theorem 23 and Proposition
18]{RoByci}
\\
Let $\Delta, \Delta'$ be two lobes of $S\Gamma(X,R\cup W;w)$
colored respectively by $i, 3-i$ for some $i = 1, 2$ with
$\Delta\mapsto\Delta'$. Let $\nu\in V(\Delta)\cap V(\Delta')$,
$f\in E(U)$ such that $(\nu,\Delta',\nu)\simeq
(x\rho,S\Gamma(X_{3-i}, R_{3-i};f)/\rho,x\rho)$ where $\rho$ is
the least equivalence relation on $S\Gamma(X_{3-i}, R_{3-i};f)$
which identifies the net $N(x, S\Gamma(X_{3-i}, R_{3-i};f))$ and
makes $S\Gamma(X_{3-i}, R_{3-i};f)/\rho$ deterministic. Then
$$
(\nu,S\Gamma(X,R\cup W;f),\nu)\simeq (x\Xi,S\Gamma(X,R\cup
W;f)/\Xi,x\Xi)
$$
where
$$
\Xi\subseteq V(S\Gamma(X,R\cup W;f))\times
V(S\Gamma(X,R\cup W;f))
$$
is defined by: $q\Xi q'$ if there are $y,y'\in N(x,S\Gamma(X_{3-i}, R_{3-i};f))$ and
$t\in(X\cup X^{-1})^*$ such that $(y,t,q)$ and $(y',t,q')$ are
paths in $S\Gamma(X,R\cup W;f)$.
\\
Moreover the following lifting property for $\Xi$ holds: if $(p\Xi,h,q\Xi)$ is a
path in $S\Gamma(X,R\cup W;f)/\Xi$ then for each $p\in
p\Xi$ there is a path $(p,h,q')$ in $S\Gamma(X,R\cup W;f)$ with
$q'\in q\Xi$.
\end{prop}
We have the following characterization.

\begin{theorem}\label{theo:characterization of more than one host}
Let $[S_1,S_2;U]$ be an amalgam of finite inverse semigroups, let
$w\in (X\cup X^{-1})^+$. The following are equivalent:
\begin{enumerate}
    \item $S\Gamma(X,R\cup W;w)$ has more than one host.
    \item Each host of $S\Gamma(X,R\cup W;w)$ is the \sch graph of some idempotent
    of $U$ relative to the presentation $\langle X_i|R_i\rangle$ of $S_i$ for some $i\in\{1,2\}$.
    \item $ww^{-1}\mathcal{D}^{S_1*_US_2} f$ for some idempotent $f\in E(U)$.
\end{enumerate}
\end{theorem}
\begin{proof}
$1)\Rightarrow 3)$. Assume that $S\Gamma(X,R\cup W;w)$ has more
than one host. Then by Proposition \ref{Prop:host set} there are
(at least) two adjacent lobes of $S\Gamma(X,R\cup W;w)$ which are
hosts. Let $\nu$ be an intersection vertex between these adjacent
hosts $\Delta_i=\Delta_i(\nu)$, $i\in \{1,2\}$ and let
$f=f(e_1(\nu))$. By definition of host $cl_{R\cup
W}(\Delta_2)=S\Gamma(X,R\cup W;w)$. Moreover
$\Delta_1\mapsto\Delta_2$, so by Proposition \ref{prop:lifting}
$(\nu,S\Gamma(X,R\cup W;w),\nu)\simeq (x\Xi,S\Gamma(X,R\cup
W;f)/\Xi,x\Xi)$.  Let $e$ be an idempotent labelling a loop based
at $\nu$ in $S\Gamma(X,R\cup W;w)$ then again by Proposition
\ref{prop:lifting}, $e$ labels also a loop based at $x$ in
$S\Gamma(X,R\cup W;f)$, hence  $e\ge f$. So $f$ is the minimum
idempotent labelling a loop based at $\nu$ in $S\Gamma(X,R\cup
W;w)$ whence $S\Gamma(X,R\cup W;w)=S\Gamma(X,R\cup
W;ww^{-1})\simeq S\Gamma(X,R\cup W;f)$, then by Proposition
\ref{prop:old Ste} $ww^{-1}\mathcal{D}^{S_1*_US_2} f$.\\
$3)\Rightarrow 2)$. Put $\Delta_i=S\Gamma(X_{i},R_{i};f)$ for
$i=1,2$. By Proposition \ref{prop:old Ste} $S\Gamma(X,R\cup
W;ww^{-1})\simeq S\Gamma(X,R\cup W;f)$. Obviously $S\Gamma(X,R\cup
W;f)$ is obtained by iterated applications of Construction 5 to
$\Delta_1$, then $\Delta_1$ is a host of
$S\Gamma(X,R\cup W;w)$. Now let $\Delta$ be any host of
$S\Gamma(X,R\cup W;w)$ and assume that it is colored $j\in
\{1,2\}$. We prove that $\Delta$ is a \sch graph of some
idempotent of $U$ by induction on the length $n$ of the reduced
lobe path connecting $\Delta_1$ to $\Delta$. If $n=0$ the
statement is trivially true. So let $P:\Delta_1,\Delta_2,\ldots,\Delta_n=\Delta$ be the reduced lobe path connecting $\Delta_1$ with $\Delta$. Since $\Delta_1$ and
$\Delta_{n}=\Delta$ are hosts, by Proposition \ref{Prop:host set}
$\Delta_{n-1}$ is a host and by induction hypothesis it is a \sch
graph of some idempotent of $U$. Let $\nu \in V(\Delta_{n-1})\cap
V(\Delta_n)$. Since $\Delta_{n-1}\mapsto\Delta_{n}$, then by Proposition \ref{prop:hosts which are sch} 
$e_{3-j}(\nu)=f(e_j(\nu))\in E(U)$. Hence, since $\Delta_{n-1}$
is a \sch graph, $(\nu,\Delta_{n-1},\nu)\simeq
\mathcal{A}(X_{3-j},R_{3-j};f(e_j(\nu)))$ and so by Proposition \ref{prop:hosts which are sch} $\Delta_n\simeq S\Gamma(X_{j},R_{j};f(e_j(\nu)))$.
\\
$2)\Rightarrow 1)$. Let $\Delta$ be a host colored $i$ of
$S\Gamma(X,R\cup W;w)$. Then $\Delta\simeq S\Gamma(X_i,R_i;f)$ for
some $f\in E(U)$. Then $f=e_i(\nu)=f(e_i(\nu))$ for some $\nu\in
V(\Delta)$. Applying Construction 5 at $\nu$ one gets a new lobe
$\Delta'$ such that $\Delta'\mapsto\Delta$ by Proposition
\ref{prop:hosts which are sch}. Let $\Lambda$ be any lobe of
$S\Gamma(X,R\cup W;w)$. Since $\Delta$ is a host, then $\Lambda$
feeds off $\Delta$ that in turns directly feeds off $\Delta'$. So
$\Lambda$ feeds off $\Delta'$ and $\Delta'$ is a host.
\end{proof}

In the sequel for a \sch graph $S\Gamma(X,R\cup W;w)$ we denote
$Host(S\Gamma(w))$ the union of all its hosts. We
characterize \sch graphs $S\Gamma(X,R\cup W;e)$, such that
$Host({S\Gamma}(e))$ is an infinite graph. Since a host has
finitely many finite lobes, these opuntoids need to have infinitely
many hosts, hence by Theorem \ref{theo:characterization of more than one host}
we necessarily have $e\mathcal{D}^{S_1*_US_2} f$ for some idempotent
$f\in E(U)$. By the above Theorem in such case all hosts are
lobes which are \sch graphs of some idempotents of $U$ relative to the
presentation $\langle X_i|R_i\rangle$ of $S_i$ for some
$i\in\{1,2\}$.

\begin{defn}\label{defn:directisomorphic}
  Let $\Delta,\Delta'$ be two lobes of an opuntoid graph such that
  $\Delta'=\phi(\Delta)$ for some isomorphism $\phi$. Let $\Delta=\Delta_0,\Delta_1,\ldots,\Delta_n=\Delta'$
  be the reduced lobe path connecting $\Delta$ to $\Delta'$ and let $\nu_1\in V(\Delta_0)\cap
  V(\Delta_1)$. The isomorphism $\phi$ is called a shift-isomorphism (and $\Delta, \Delta'$ are called shift-isomorphic by $\phi$) if $\phi(\nu_1)\notin V(\Delta_{n-1})\cap V(\Delta_n)$.\\
  The lobes $\Delta,\Delta'$ are called successive isomorphic lobes
  if no $\Delta_i$ ($0<i<n$) is isomorphic to $\Delta_0$.
\end{defn}
We have the following lemma.
\begin{lemma}\label{lem:notshift-isomorphism} Let $e\in E(S_1*_U S_2)$ with $S_1,S_2$
finite inverse semigroups and let $e\mathcal{D}^{S_1*_US_2} f$ for
some idempotent $f\in E(U)$. Let $\Delta, \Delta'$ be two distinct
lobes of $Host({S\Gamma}(e))$ colored $i$ such that
$\Delta'=\phi(\Delta)$ for some $\phi\in
Aut(Host({S\Gamma}(e)))$. Let
$\Delta=\Delta_0,\Delta_1,\ldots,\Delta_n=\Delta'$ be the reduced
lobe path connecting $\Delta$ to $\Delta'$. Then either for some
$j$ with $0\leq j \leq \lfloor \frac{n}{2}\rfloor$
$\phi|_{\Delta_j}$ is a shift-isomorphism or, for all $j$ with
$0\leq j \leq \lfloor \frac{n}{2}\rfloor$,
$\phi(\Delta_j)=\Delta_{n-j}$.
\end{lemma}
\begin{proof}
By Proposition \ref{Prop:host set} and Theorem \ref{theo:characterization of more than one host}, each $\Delta_h$,
$0\leq h\leq n$, is a host which is the \sch graph of some idempotent
of $U$ relative to the presentation $\la X_i|R_i\ra$ for some
$i=1,2$. We prove the statement by induction on $n\geq 2$.
The base of induction is trivial. If $\phi|_{\Delta_0}$ is a
shift-isomorphism of $\Delta_0$ onto $\Delta_n$ the statement is
trivially true. So assume that $\phi|_{\Delta_0}$ is not a shift-isomorphism. Let $\nu_1\in V(\Delta_0)\cap V(\Delta_1)$, then
$\phi(\nu_1)\in V(\Delta_{n-1})\cap V(\Delta_n)$. Moreover by
Proposition \ref{prop:hosts which are sch} we get
$f_1=e_i(\nu_1)\in E(U)$. Hence $e_i(\phi(\nu_1))=f_1 \in E(U)$, and so
$(\phi(\nu_1),\Delta_{n-1},\phi(\nu_1))\simeq
\mathcal{A}(X_{3-i},R_{3-i};f_1)\simeq (\nu_1,\Delta_1,\nu_1)$,
$\phi(\Delta_1)=\Delta_{n-1}$. Since the reduced lobe path from
$\Delta_1$ to $\Delta_{n-1}$ has length $n-1$ the statement holds
by induction hypothesis.
\end{proof}

\begin{prop}\label{prop:characterization of infinite host set}
Let $e\in E(S_1*_U S_2)$ with $S_1,S_2$ finite inverse semigroups
and let $e\mathcal{D}^{S_1*_US_2} f$ for some idempotent $f\in
E(U)$. Then the following are equivalent
\begin{enumerate}
    \item $Host({S\Gamma}(e))$ is infinite;
    \item $Host({S\Gamma}(e))$ has infinitely many lobes;
    \item There are two isomorphic hosts of $S\Gamma(e)$ which
are not successive isomorphic lobes;
    \item There is a shift-isomorphism between two hosts of
    $S\Gamma(e)$.
\end{enumerate}
\end{prop}
\begin{proof}
The equivalence between 1) and 2) is trivial.\\ $1)\Rightarrow
3)$. By Theorem \ref{theo:characterization of more than one host} each lobe of $Host({S\Gamma}(e))$ is a \sch graph of some
idempotent of $U$ relative to the presentation $\langle
X_i|R_i\rangle$ for some $i\in \{1,2\}$. Since $S_1,S_2$ are
finite, there are finitely many \sch graphs of idempotents of $U$
relative to the presentations $\langle X_i|R_i\rangle$ with $i\in
\{1,2\}$. Since all the lobes are finite, each lobe has finitely many adjacent lobes and so the degree of each vertex of the lobe tree
$T(Host({S\Gamma}(e)))$ is finite. Therefore there is an infinite reduced lobe path in $T(Host({S\Gamma}(e)))$ in which there are at least three isomorphic lobes, whence there are two isomorphic hosts which are not successive.\\
$3)\Rightarrow 4)$. Suppose that $S\Gamma(e)$ has two isomorphic
hosts $\Delta$ and $\Delta'$ which are not successive isomorphic
lobes. Thus, in the reduced lobe path
$\Delta=\Delta_0,\Delta_1,\ldots,\Delta_n=\Delta'$ connecting them, there is a lobe
$\Delta_h$ with $1\leq h\leq n-1$  isomorphic to both $\Delta$ and
$\Delta'$. We can assume without loss of generality that no $\Delta_j$
with $1\leq j\leq n-1, j\neq h$ is isomorphic to $\Delta$. Since
isomorphic lobes have the same color then $n\geq 4$ is even and let $t=n/2$. Let $\phi$ be the isomorphism sending $\Delta$ onto $\Delta'$. By Propositions
\ref{Prop:ExtensionToIsomorphism} and \ref{prop:IsomorphismOpuntoidHost} $\phi$ can be extended to an
automorphism $\overline{\phi}\in Aut(Host({S\Gamma}(e)))$.
Assume that for each lobe $\Lambda$ of $Host({S\Gamma}(e))$
$\overline{\phi}|_{\Lambda}$ is not a shift-isomorphism of
$\Lambda$ onto some host $\Lambda'$. Then by Lemma
\ref{lem:notshift-isomorphism} for all $j$ with $0\leq j<t$,
$\overline{\phi}(\Delta_j)=\Delta_{2t-j}$. Then $t=h$, otherwise
both $\Delta_h$ and $\Delta_{n-h}$ would be isomorphic to $\Delta$, hence $\overline{\phi}|_{\Delta_t}\in Aut(\Delta_t)$. Let $\nu_t\in
V(\Delta_{t-1})\cap V(\Delta_t)$. Then $\overline{\phi}(\nu_t)\in
V(\Delta_t)\cap V(\Delta_{t+1})$. Now let
$\psi:\Delta\rightarrow\Delta_t$ be an isomorphism. If $\psi$ is a
shift-isomorphism then we are done, otherwise if $\nu_1\in V(\Delta_0)\cap V(\Delta_{1})$, then $\psi(\nu_1)\in
V(\Delta_t)\cap V(\Delta_{t-1})$. Using the fact that $\psi$ preserves labeling, it is easy to see that $\psi$ is actually a bijection between the two sets $V(\Delta_0)\cap V(\Delta_{1})$ and $V(\Delta_t)\cap V(\Delta_{t-1})$. Thus let $\nu'$ be the vertex of $V(\Delta_0)\cap V(\Delta_{1})$ such that $\psi(\nu')=\nu_{t}\in
V(\Delta_{t-1})\cap V(\Delta_t)$. Hence $\overline{\phi}(\psi(\nu'))=\overline{\phi}(\nu_{t})\in V(\Delta_t)\cap V(\Delta_{t+1})$, i.e. $\overline{\phi}(\psi(\nu'))\notin V(\Delta_{t-1})\cap V(\Delta_{t})$. Hence
the map $\psi\cdot\overline{\phi}:\Delta\rightarrow \Delta_t$
(defined by
$\psi\cdot\overline{\phi}(\nu)=\overline{\phi}(\psi(\nu)))$ is a
shift-isomorphism from $\Delta$ to $\Delta_t$.\\
$4)\Rightarrow 2)$. Assume by contradiction that
$Host({S\Gamma}(e))$ has two shift-isomorphic lobes, and
finitely many lobes. Let $\Delta, \Delta'$ be two shift-isomorphic lobes by $\phi$ in $Host({S\Gamma}(e))$ such that
the reduced lobe path
$\Delta=\Delta_0,\Delta_1,\ldots,\Delta_n=\Delta'$ from $\Delta$
to $\Delta'$ is of maximal length. By Proposition \ref{Prop:host
set} and Theorem \ref{theo:characterization of
more than one host} each $\Delta_j$, $0\leq j\leq n$ is a host and the \sch graph
of some idempotent of $U$. Moreover by Proposition \ref{prop:hosts which are sch}, $(\nu_{j},\Delta_{j-1},\nu_{j})\simeq
\mathcal{A}(X_i,R_i;f_j)$, $(\nu_{j},\Delta_{j},\nu_{j})\simeq
\mathcal{A}(X_{3-i},R_{3-i},f_j)$ where $\nu_j\in
V(\Delta_{j-1})\cap V(\Delta_j)$, $f_j=e_i(\nu_{j})=e_{3-i}(\nu_j)
\in U$ and $i\in \{1,2\}$ is the color of $\Delta_{j-1}$. By Propositions \ref{Prop:ExtensionToIsomorphism} and \ref{prop:IsomorphismOpuntoidHost} the isomorphism $\phi:\Delta\rightarrow\Delta'$ can be extended to an automorphism $\overline{\phi}\in Aut(Host({S\Gamma}(e)))$. We prove by induction on $h$ that, for all $h$ with $0\leq h\leq n$, $\overline{\phi}$ maps the subopuntoid subgraph
$\Theta_h=\bigcup_{0\leq j\leq h}\Delta_j$ of
$Host({S\Gamma}(e))$ onto a subopuntoid subgraph of
$Host({S\Gamma}(e))$ whose lobes, except eventually $\Delta_n$,
are all different from the lobes of $\Theta_h$. Moreover
$\overline{\phi}$ is a shift-isomorphism between $\Delta_h$ and
$\overline{\phi}(\Delta_h)$. The base of induction is trivial. So
let $\Theta_{h-1}=\bigcup_{0\leq j\leq h-1}\Delta_j$ and put
$\Delta_{n+j}=\overline{\phi}(\Delta_{j})$ for all $0\leq j\leq
h-1$. By Lemma \ref{lemma:InducedAuthomorphismOnTree}
$\overline{\phi}(\Theta_{h-1})=\bigcup_{0\leq j\leq
h-1}\Delta_{n+j}$  and for all $j$ with $0\leq j\leq h-2$,
$\Delta_{n+j}$ is adjacent to $\Delta_{n+j+1}$. Moreover by
induction hypothesis $\Theta_{h-1}$ and
$\overline{\phi}(\Theta_{h-1})$ have disjoint sets of lobes and
$\overline{\phi}$ is a shift-isomorphism of $\Delta_{h-1}$ onto
$\Delta_{n+h-1}$. Let $\nu_h$ be an intersection vertex between
$\Delta_{h-1}$ and $\Delta_h$ and let $i\in\{1,2\}$ be the color of
$\Delta_{h-1}$. By Lemma \ref{lemma:ExtensionProperty}
$\overline{\phi}$ maps $\Theta_{h}=\Theta_{h-1}\cup \Delta_h$,
onto $\overline{\phi}(\Theta_{h-1})\cup S\Gamma(X_{3-i},R_{3-i};
f(e_i(\overline{\phi}(\nu_h))))$. So $S\Gamma(X_{3-i},R_{3-i};
f(e_i(\overline{\phi}(\nu_h))))=\Delta_{n+h}$ does not coincide
with any lobe of $\Theta_h$. Moreover  if $\nu_{h+1}$ is an
intersection vertex between $\Delta_h$ and $\Delta_{h+1}$ then
$\overline{\phi}(\nu_{h+1})\notin V(\Delta_{n+h-1})\cap
V(\Delta_{n+h})$ so $\overline{\phi}$ is a shift isomorphism between $\Delta_h$ and $\Delta_{n+h}$. In particular for $h=n$, $\overline{\phi}$ is a
shift-isomorphism of $\Delta_n$ onto $\Delta_{2n}$ and
$\overline{\phi}^2$ is a shift-isomorphism of $\Delta_0$ onto
$\Delta_{2n}$, against the assumption that the reduced lobe path
connecting $\Delta=\Delta_0$ to $\Delta'=\Delta_n$ is a path of
maximal length among the reduced lobe paths connecting two hosts
which are isomorphic under a shift-isomorphism. Then
$Host({S\Gamma}(e))$ has infinitely many lobes.
\end{proof}

From the above proposition we derive the following corollary.
\begin{cor}\label{cor:characterization of infinite host set}
Let $e\in E(S_1*_U S_2)$ with $S_1,S_2$ finite inverse semigroups
and let $e\mathcal{D}^{S_1*_US_2} f$ for some idempotent $f\in
E(U)$. Then $Host({S\Gamma}(e))$ is infinite if and only if in
$Host({S\Gamma}(e))$ there is a reduced lobe path
$P:\Delta_{0},\ldots,\Delta_{t-1},\Delta_{t},\Delta_{t+1},\ldots\Delta_{2t}$
with $\Delta_0\simeq\Delta_{t}\simeq\Delta_{2t}$.
\end{cor}
\begin{proof}
  By Proposition \ref{prop:characterization of infinite host set} if $Host({S\Gamma}(e))$ is infinite
  then there is a shift-isomorphism between two
  lobes of $Host({S\Gamma}(e))$. Let $\Delta$ and $\Delta'$ be such lobes
  and let $\phi:\Delta\rightarrow\Delta'$
  be a shift-isomorphism. If $P$ is the reduced lobe path from
  $\Delta$ to $\Delta'$, by the same argument of the proof of implication $4)\Rightarrow 2)$ of Proposition \ref{prop:characterization of infinite host set}, $P\cup\phi(P)$ is a reduced lobe path in $Host({S\Gamma}(e))$ satisfying the statement. Conversely, let $P:\Delta_{0},\ldots,\Delta_{t-1},\Delta_{t},\Delta_{t+1},\ldots\Delta_{2t}$
  with $\Delta_0\simeq\Delta_{t}\simeq\Delta_{2t}$ be a reduced lobe path of
  $Host({S\Gamma}(e))$, then in $P$ there are two non successive isomorphic lobes and
  so $Host({S\Gamma}(e))$ is infinite by Proposition \ref{prop:characterization of infinite host
  set}.
\end{proof}

\section{Review of the Bass-Serre theory}\label{Sect.BS}

For the sake of completeness, we shortly review
the Bass-Serre theory of groups acting on
graphs. We refer the reader interested in more details to
\cite{GroupActingOnGraphs,Serre}.\\
Let $G$ be a group, and let $X=(Vert(X),Edge(X))$ be a graph with
initial and terminal vertex maps $\alpha, \omega:
Edge(x)\rightarrow Vert(X)$. If the action $\cdot$ of $G$ on the
set $Vert(X)\cup Edge(X)$ satisfies for all $g\in G, y\in Edge(X)$
the conditions: $\alpha(g\cdot y)=g\cdot\alpha(y)$, $\omega(g\cdot
y)=g\cdot\omega(y)$ and $g\cdot \overline{y}=\overline{g\cdot y} $
(where $\overline{y}$ denotes the opposite edge of $y$), we say
that $G$ \emph{acts on the graph} $X$. The action of $G$ on the
graph $X$ is \emph{without inversions} if $g\cdot y\neq
\overline{y}$ for all $y\in Edge(X), g\in G$. If $G$ acts on a
graph $X$ without inversions, the \emph{quotient graph} of the
action of $G$ on $X$ is the graph $G\backslash X$ whose vertex and
edge sets are respectively the set $\{G\cdot v|v\in Vert(X)\}$ of
the orbits of $G$ of the vertices of $X$ and the set $\{G\cdot
y|y\in Edge(X)\}$ of the orbits of $G$ of the edges of $X$, with
the incidence relation defined by $\alpha(G\cdot
y)=G\cdot\alpha(y),\ \omega(G\cdot y)=G\cdot\omega(y)$ and
$\overline{G\cdot y}=G\cdot \overline{y}$ for all $y\in Edge(X)$.
The map $v\rightarrow G\cdot v$, $y\rightarrow G\cdot y$, for all
$v\in Vert(X),\ y\in Edge(X)$ is a map of the graph $X$ onto
$G\backslash X$ and when $X$ is a connected graph the subtrees of
$G\backslash X$ lift to subtrees of
$X$ \cite{Serre}.\\
Let $X$ be a connected non-empty graph and let $\mathcal{G}$ be a
mapping assigning to each $v\in Vert(X)$ and to each $y\in
Edge(X)$ a group $G_v$ and $G_y$ so that $G_y=G_{\overline{y}}$.
Assume that for each $y\in Edge(X)$ there are two group
monomorphisms $\sigma_y:G_y\rightarrow G_{\alpha(y)}$ , $\tau_y:
G_y\rightarrow G_{\omega(y)}$ such that
$\sigma_y=\tau_{\overline{y}}$. Then $X$ with the group assignment
$\mathcal{G}$ is called a \emph{graph of groups}
$(\mathcal{G}(-),X)$. Let $T$ be any maximal subtree of $X$, the
\emph{fundamental group} $\pi(\mathcal{G}(-),X,T)$ of
$(\mathcal{G}(-),X)$ with respect to $T$ is generated by the
disjoint union of vertex groups $G_v,\ v\in Vert(X)$ and by the
edges $Edge(X)$, subject to the relations $\{\overline{y}=y^{-1},
y^{-1}\sigma_y(a)y=\tau_y(a)|y\in Edge(X), a\in
G_y\}\cup\{y=1|y\in T\}$. The fundamental group of a graph of
groups is, up to isomorphisms, independent on the choice of $T$
and so it will be denoted by $\pi(\mathcal{G}(-),X)$. From the
structure theorem of Bass-Serre theory we know that
\begin{prop}\label{Prop:embedding of vertex groups}
Let $(\mathcal{G}(-),X)$ be a graph of groups. Then each vertex
group $G_v$ and edge group $G_y$ is embedded in the fundamental
group $\pi(\mathcal{G}(-),X)$.
\end{prop}
Two graphs of groups $(\mathcal{G}(-),X)$, $(\mathcal{H}(-),Y)$
are said to be \emph{isomorphic} if there is a graph isomorphism $\Phi:X\rightarrow
Y$ together with a collection of group isomorphisms $\Phi_v:G_v\rightarrow
H_{\Phi(v)}$, $\Phi_y:G_y\rightarrow H_{\Phi(y)}$ satisfying the
conditions $\Phi_y=\Phi_{\overline{y}}$,
$\Phi_{\alpha(y)}\sigma_y=\sigma_{\Phi(y)}\Phi(y)$ and
$\Phi_{\omega(y)}\tau_y=\tau_{\Phi(y)}\Phi(y)$. It is easy to see
that isomorphic graphs of groups have isomorphic fundamental
groups.\\
A graph of groups $(\mathcal{H}(-),X)$ is \emph{conjugate} to
$(\mathcal{G}(-),X)$ if it has the same group assignments of
$(\mathcal{G}(-),X)$, the same embedding $\sigma_y$ and whose
embedding $\tau_y$ are the ones of $(\mathcal{G}(-),X)$ followed
by a conjugation by an element of $G_{\omega(y)}$. A graph of
groups $(\mathcal{H}(-),Y)$ is \emph{conjugate isomorphic} to the
graph of groups $(\mathcal{G}(-),X)$ if it is isomorphic to a
conjugate of $(\mathcal{G}(-),X)$. Two conjugate isomorphic graphs
of groups have the same fundamental groups.\\
Now we outline how to construct a graph of groups starting from
the action of a group $G$ on a connected non-empty graph $X$. Let
$Y=G\backslash X$ and let $A$ be an orientation of the edges of
$Y$, i.e a subset of edges of $Y$ containing exactly one edge for
each pair of opposite edges in $Y$. Let $T$ be a maximal subtree of
$Y$ and $T'$ its lifting to $X$, for each $v\in Vert(T)=Vert(Y),
y\in Edge(T)$, we denote by $j(v),j(y)$ the lift of $v$ and $y$ in
$T'$. For each $v\in Vert(T)$ and  $y\in Edge(T)$ we put
$G_v=Stab_G(j(v))$ and $G_y=Stab_G(j(y))$.  Since
$Stab_G(j(y))\subseteq Stab_G(j(\alpha(y)))$ and
$Stab_G(j(y))\subseteq Stab_G(j(\omega(y)))$, then for each $y\in
Edge(T)$ the monomorphisms $\sigma_y, \tau_y$ are the
inclusions. Now let $y\in (Edge(Y)-Edge(T))\cap A$ and let
$x=j(y)\in Edge(X)$ an edge mapping in $y$ such that
$\alpha(x)=j(\alpha (x))$. Again we put $G_y=Stab_G(j(y))$ and
$\sigma_y$ is the inclusion of $G_y$ in $G_{\alpha(x)}$. Moreover
since $\omega(x)$ and $j(\omega(y))$ belong to the same vertex
orbit of $G$ there exists $g_y\in G$ mapping $\omega(x)$ to
$j(\omega(y))$ and the two groups $Stab_G(\omega(x))$,
$Stab_G(j(\omega(y)))$ are conjugate in $G$ by $g_y$. So we set
$\tau_y=g_y\cdot\iota\cdot g_y^{-1}$ where $\iota$ is the
inclusion of $Stab_G(j(y))$ in $Stab(\omega(x))$. Then we put for
each edge $y\in Edge(Y)-A$  $G_y=G_{\overline{y}}
=Stab_G(\overline{y})$, $\sigma_y=\tau_{\overline{y}}$,
$\tau_y=\sigma_{\overline{y}}$, completely constructing a graph of
groups $(\mathcal{G}(-),Y)$. Its fundamental group
$\pi(\mathcal{G}(-),Y)$ is homomorphic on G by the group
homomorphism $\Phi$ defined by the inclusions $G_v\rightarrow G$
and the mapping $\Phi(y)=g_y$ where $g_y$ is the element of $G$
mapping $\omega(x)$ on $j(\omega(y))$. The structure theorem of
Bass-Serre theory says
\begin{theorem}\label{structure theorem}
Let $G$ be a group acting without inversions on a connected graph
$X$. Then $X$ is  tree if and only if
$\Phi:\pi(\mathcal{G}(-),G\backslash X)\rightarrow G$ is an
isomorphism.
\end{theorem}

\section{Maximal subgroups}\label{sec: maximal sub}

Let $e\in E(S_1*_US_2)$, the lobe graph $\mathcal{T}_e =\mathcal{T}(Host({S\Gamma}(e)))$ can be
seen as a directed tree $(Vert(\mathcal{T}_e),Edge(\mathcal{T}_e),\alpha,\omega)$ where
$Vert(\mathcal{T}_e)$ is the set of lobes of union of all hots $Host({S\Gamma}(e))$. $Edge(\mathcal{T}_e)$ is formed
by the pairs of adjacent lobes of $Host({S\Gamma}(e))$, oriented so that,
for each $y=(\Delta,\Delta')\in Edge(\mathcal{T}_e)$, $\alpha(y)$ is the lobe
$\Delta$ colored $1$ and $\omega(y)$ is the lobe $\Delta'$
colored $2$. Lemma \ref{lemma:InducedAuthomorphismOnTree} and Proposition \ref{prop:IsomorphismOpuntoidHost} prove that the group
$H_e^{S_1*_U S_2}\simeq Aut(S\Gamma(e))$ acts on $\mathcal{T}_e$ without inversions. Thus we can build a graph of groups $(\mathcal{G}(-),G\backslash
\mathcal{T}_e)$ starting from the action of the group $G=H_e^{S_1*_U S_2}$ on the
connected non-empty graph $\mathcal{T}_e$ as
described in the previous section.  As in Section \ref{Sect.BS}, the quotient graph $G\backslash
\mathcal{T}_e$ will be denoted by $Y$.  Theorem \ref{structure
theorem} gives us immediately  the following:

\begin{cor}\label{prop:MaximalSubgroupIsomFundamentalGroup}
Let $e\in E(S_1*_US_2)$ be an idempotent in the amalgamated free
product of two finite inverse semigroups $S_1,S_2$.  Let  $Y=H_e^{S_1*_U S_2}\backslash
\mathcal{T}_e$. Then
$$
H_e^{S_1*_U S_2}\simeq \pi(\mathcal{G}(-),Y).
$$
\end{cor}

To study the structure of maximal subgroups  in more details, we analyze the two following different situations:

\begin{itemize} 
    \item [-]{\bf Case 1:} 
    $e$ is an ``old idempotent'': $e$ is $\mathcal{D}^{S_1*_US_2}$-related to some idempotent of $S_1$ or $S_2$.
    \item [-]{\bf Case 2:}
    $e$ is a ``new idempotent'': $e$ is not $\mathcal{D}^{S_1*_US_2}$-related to any idempotent of $S_1$ or
    $S_2$.
\end{itemize}

\subsection{Case 1}

Although in general the lobes of \sch graphs of elements of amalgams of finite inverse semigroups are only DV-quotients,  if we restrict our attention to hosts of \sch graphs  of original idempotents the situation is nicer:

\begin{theorem}\label{theo:nomultiedge}
Let $e$ be an idempotent in $S_1$ or $S_2$. With the above
notations,  let  $Y=H_e^{S_1*_U S_2}\backslash
\mathcal{T}_e$. Then
\begin{enumerate}
    \item each $\Delta \in Vert(Y)$ is a \sch graph $S\Gamma(X_i,R_i;e_i(\nu))$ for some $\nu\in Vert(\Delta)$;
    \item $Y$ is finite;
    \item $(\Delta_1,\Delta_2)\in Edge(Y)$ if and only if the following
    conditions hold
    \begin{itemize}
        \item $e$ is $\mathcal{D}^{S_1*_US_2}$-related to some idempotent
        of $U$;
        \item $\Delta_i(\nu)\simeq S\Gamma(X_i,R_i;f)$ for some $f\in
        E(U)$;
        \item there is a lobe $\Delta_2'$ of
        $Host({S\Gamma}(e))$ such that $(\Delta_1,\Delta_2')\in
        Edge(\mathcal{T}_e)$, $\psi(\Delta_2')=\Delta_2$, for
        some automorphism $\psi\in Aut(Host({S\Gamma}(e)))$, and $e_i(\nu')=f$ for
        intersection vertex $\nu'$ between $\Delta_1$ and $\Delta_2'$.
    \end{itemize}
\end{enumerate}
\end{theorem}
\begin{proof}
Assume that $e$ is an idempotent of $S_1$. If we start from a word $u\in (X_{1}\cup X_{1}^{-1})^{*}$ equivalent to $e$ in $S_{1}$, then it is clear that the underlying graph $\Delta_{0}$ of $Core(u)$ is isomorphic to $S\Gamma(X_1,R_1;u)$. Thus $\Delta_{0}$ is a host of $\Gamma=S\Gamma(X,R\cup W;e)$. If $e$ is not $\mathcal{D}^{S_1*_US_2}$-related to any idempotent of $U$ then it is the unique host, otherwise $\Gamma$ has more than one host, each host is a lobe and a \sch graph of some idempotent of $U$ relative to the presentation $\la X_i|R_i\ra$,
for some $i\in\{1,2\}$ by Theorem \ref{theo:characterization of more than one host}, so statement 1 is proved.
\\
If $e$ is not $\mathcal{D}^{S_1*_US_2}$-related to any idempotent
of $U$ then $Y$ is trivially finite. So assume that $e$ is
$\mathcal{D}^{S_1*_US_2}$-related to some idempotent of $U$. Since by Theorem \ref{theo:characterization of more than one host}
all hosts of $S\Gamma(e)$ are \sch graphs of idempotents of $U$,
there are  at most $|U|$ different hosts up to isomorphisms. From
Propositions \ref{Prop:ExtensionToIsomorphism}, \ref{prop:IsomorphismOpuntoidHost} it follows that two
isomorphic hosts lie in the same $H_e^{S_1*_U
S_2}$-orbit, and so statement 2 holds.
\\
Let us prove statement 3. The ``if" part is trivial. We prove the
``only if" part. Let $(\Delta_1,\Delta_2)\in Edge(Y)$, then
$\Gamma$ has more than one host, hence by Theorem \ref{theo:characterization of more than one host}, $e$ is
$\mathcal{D}^{S_1*_US_2}$-related to some idempotent of $U$.
Moreover there is $(\Delta_1',\Delta_2')\in
Edge(\mathcal{T}_e)$ which lies in the same $H_e^{S_1*_U
S_2}$-orbit of $(\Delta_1,\Delta_2)$. The lobes
$\Delta_1',\Delta_2'$ are adjacent, feed off each other and share
at least a common vertex $q$. By Theorem
\ref{theo:characterization of more than one host} $\Delta_i'\simeq
S\Gamma(X_i,R_i;e_i(q))$ with $e_i(q)\in E(U)$ for $i=1,2$.  Moreover
there is an automorphism $\varphi$ of $Host({S\Gamma}(e))$ such that
$\varphi(\Delta_1')=\Delta_1$, then $\Delta_1$ and
$\varphi(\Delta_2')$ are adjacent lobes that share the vertex
$\varphi(q)=\nu'$. Hence $e_i(\nu')=e_i(q)\in E(U)$.
$\varphi(\Delta_2')\simeq S\Gamma(X_2,R_2;e_2(\nu'))$ is a host and
$(\Delta_1,\varphi(\Delta_2'))\in Edge(\mathcal{T}_e)$
lies in the same $H_e^{S_1*_U S_2}$-orbit of
$(\Delta_1,\Delta_2)$. Then there is an automorphism $\psi$ of
$Host({S\Gamma}(e))$ such that $\psi(\varphi(\Delta_2'))=\Delta_2$,
$\psi(\Delta_1)=\Delta_1$ and $\psi(\nu')=\nu$, hence $e_1(\nu)\in
E(U)$ and $\Delta_i\simeq S\Gamma(X_i,R_i;f)$ with $f=e_1(\nu)$.
\end{proof}
As an immediate consequence of the  previous proof, we get that the graph of groups, in the case when $e$ is an original idempotent but it is not $\mathcal{D}$-related in the amalgam to any idempotent in $U$,  consists of just a single vertex. This vertex corresponds to the \sch graph of some $g$ in $S_i$ and we get by, Proposition \ref{prop:old Ste}, the following description of maximal subgroups in this situation.

\begin{cor}\label{MaxNotD-relatedtoU}
Let $e$ be $\mathcal{D}$-related in $S_1*_US_2$ to some idempotent $g$
of $S_i \setminus U$,  $i=1,2$ , then  the maximal subgroup $H^{S_1*_US_2}_e \simeq H^{S_i}_g$.
\end{cor}
Let $e$ be $\mathcal{D}$-related in $S_1*_US_2$ to some idempotent
of $U$, then by Theorem \ref{theo:characterization of more than
one host} the \sch graph $S\Gamma(X,R\cup W;e)$ has more than one host and each host is a lobe
and a \sch graph of some idempotent of $U$ relative to the
presentation $\la X_i|R_i\ra$ of $S_i$, for some $i\in\{1,2\}$.
Moreover, by Lemma \ref{theo:nomultiedge} the graph $Y$ is finite and
in this case the graph of groups $(\mathcal{H}_e(-),Z_e)$ built
according to Bennett's construction can be used to describe the
structure of the maximal subgroup of $S_1*_US_2$ containing $e$.
Namely, we can prove, along the same line of the proof of Theorem 2
in \cite{Ben2} that the graph of groups $(\mathcal{G}(-),Y)$,
built starting from the action of the group $G=H_e^{S_1*_U S_2}$
on the connected non-empty graph $\mathcal{T}_e$ as
described in the Section \ref{Sect.BS}, is conjugate isomorphic to
the graph of groups $(\mathcal{H}_e(-),Z_e)$. However, the graph of groups $(\mathcal{G}(-),Y)$ gives us more information, thanks also to Theorem \ref{theo:nomultiedge} from which we can obtain a better description of the associated groups which are stabilizers of vertices and edges in $\mathcal{T}_{e}$. Indeed, by Theorem \ref{theo:nomultiedge} vertices of $\mathcal{T}_{e}$ are \sch graphs of idempotents belonging to $U$. Thus we immediately derive from Proposition \ref{prop:old Ste} and Propositions \ref{Prop:ExtensionToIsomorphism}, \ref{prop:IsomorphismOpuntoidHost} that the stabilizers of the vertices appearing in the graph of groups $(\mathcal{G}(-),Y)$ are maximal subgroups of idempotents of $U$ in the original semigroups $S_{1}$, or $S_{2}$ (depending on the color of the lobe). The next proposition gives us a description of the stabilizer of an edge in $\mathcal{T}_{e}$, but first we need the following lemma.
\begin{lemma}\label{lem:stabilazer of edge is maximal subgroup}
Let $(\nu,\Delta,\nu)=\mathcal{A}(X_k,R_k,f)$ for some
$k\in\{1,2\}$ and $f\in U$. Let $
I(\nu,\Delta)=\{y\in V(\Delta):(\nu,u,y)\:\mbox{is a path in}\;\Delta\;\mbox{for some}\; u\in U\}$. Then
$$ 
H_f^U\simeq\{\varphi\in
Aut(\Delta):\varphi(I(\nu,\Delta))\subseteq I(\nu,\Delta)\}
$$
\end{lemma}
\begin{proof}
Theorem 3.5 of \cite{Steph} shows that $H_f^k\simeq Aut(\Delta)$
by the isomorphism $m\mapsto \phi_m$ defined by
$\phi_m(v)=m^{-1}v$. Since $\phi_{k}$ is an embedding of $H_f^U$ into $H_f^k$, then $H_f^U$ also embeds into $Aut(\Delta)$. We claim that the map $u\mapsto \psi_{\phi_k(u)}$ defined by
and $\psi_{\phi_k(u)}(v)=\phi_k(u^{-1})v$ with $v\in
V(\Delta)$, $u\in H_f^U$ is an isomorphism from $H_f^U$ onto $\{\varphi\in
Aut(\Delta):\varphi(I(\nu,\Delta))\subseteq
I(\nu,\Delta)\}$. To show that
$\psi_{\phi_k(u)}\in\{\varphi\in
Aut(\Delta):\varphi(I(\nu,\Delta))\subseteq I(\nu,\Delta)\}$ it is enough to prove that
$\psi_{\phi_k(u)}(\nu)\in I(\nu,\Delta)$ since each
element of $I(\nu,\Delta)$ is connected to $\nu$ by some
element of $U$ and $\psi_{\phi_k(u)}\in Aut(\Delta)$. Since
$f$ is the unity of $H_f^U$ we get $u=fuf$, moreover since
$\nu=\phi_k(f)$ we get:
$$
\psi_{\phi_k(u)}(\nu)=\phi_k(fu^{-1}f)\phi_k(f)=
\phi_k(f)\phi_k(fu^{-1}f)=\nu\phi_k(fu^{-1}f)
$$
so by \cite{Steph} $fu^{-1}f$ labels a path connecting $\nu$ to $\psi_{\phi_k(u)}(\nu)$, whence $\psi_{\psi_k(u)}(\nu)\in I(\nu,\Delta)$. It remains to show that $u\mapsto \psi_{\phi_k(u)}$ is surjective. Let $\psi\in\{\varphi\in Aut(\Delta):\varphi(I(\nu,\Delta))\subseteq I(\nu,\Delta)\}$ then there is some $u\in U$ which labels
a path in $\Delta$ connecting $\nu$ to $\psi(\nu)$. Since $\psi$ is an
automorphism also $fuf$ labels a path connecting $\nu$ to $\psi(\nu)$. Note that the
element $fu^{-1}f\in H_f^U$ and $\nu\phi_k(fuf)=\psi(\nu)$. Thus, consider 
the automorphism $\psi_{\phi_k(fu^{-1}f)}$, then:
$$
\psi_{\phi_k(fu^{-1}f)}(\nu)=\psi_{\phi_k(fu^{-1}f)}(\phi_k(f))=\phi_k(f)\phi_k(fuf)=\nu\phi_k(fuf)=\psi(\nu)
$$
and so $\psi_{\phi_k(fu^{-1}f)}=\psi$ since they coincide on a vertex.
\end{proof}
\begin{prop}\label{prop:stabilizer of an edge is maximal subgroup of U}
Let $\Delta_1,\Delta_2$ be two adjacent lobes of $Host({S\Gamma}(e))$ and let $f=f(e_{1}(\nu))=f(e_{2}(\nu))$ for some intersection vertex $\nu\in V(\Delta_1)\cap V(\Delta_2)$. If we let $e=(\Delta_{1},\Delta_{2})$, then $Stab_{G}(e)\simeq H_f^U$ ($G=H_e^{S_1*_U S_2}$).
\end{prop}
\begin{proof}
Using Propositions \ref{Prop:ExtensionToIsomorphism} and \ref{prop:IsomorphismOpuntoidHost}, it is straightforward to check that $Stab_{G}(e)$ is isomorphic to $\{\varphi\in Aut(\Delta_1):\varphi(V(\Delta_1)\cap V(\Delta_2))\subseteq V(\Delta_1)\cap V(\Delta_2)\}$. By the assimilation property we have 
$$
V(\Delta_{1})\cap V(\Delta_{2})=\{y\in V(\Delta_{1}):(\nu,u,y)\:\mbox{is a path
in}\;\Delta_{1}\;\mbox{for some}\; u\in U\}.
$$
and so $Stab_{G}(e)\simeq H_f^U$ by Lemma \ref{lem:stabilazer of edge is maximal subgroup}.
\end{proof}
For the clarity of the presentation we record the previous facts in the following theorem.
\begin{theorem}
With the above notations, if $e$ is $\mathcal{D}^{S_{1}*_{U}S_{2}}$-related to some idempotent of $U$, then 
$$
H_e^{S_1*_U S_2}\simeq \pi(\mathcal{G}(-),Y)
$$
where $Y$ is finite. Moreover, the group $G_{v}$, $v\in Vert(Y)$, is a maximal subgroups in $S_{1}$ or $S_{2}$ of some idempotents of $U$, while $G_{e}$, $e\in Edge(Y)$, is a maximal subgroup in $U$. 
\end{theorem}

Since $Y$ is finite, from \cite[page 14]{GroupActingOnGraphs} follows that, $H_e^{S_1*_U S_2}$ is built by iteratively performing an amalgamated free-product
for each edge belonging to the maximal subtree $T$ of $Y$, followed by HNN-extensions for each edge not in $T$. Therefore the next natural steps is to characterize whether $Y$ is a tree or not. This clearly gives us a characterization which reveals whether the construction of $H_e^{S_1*_U S_2}$ involves just iterated group amalgams, or it  also involves HNN-extensions. First we characterize the case when $H_e^{S_1*_U S_2}$ is finite.

\begin{prop}
Let $e\in E(S_1*_U S_2)$ with $e \mathcal{D}^{S_1*_US_2}f$ for
some $f\in E(U)$. Then $H_e^{S_1*_U S_2}$ is finite if and only if $H_e^{S_1*_U S_2}\simeq H^{S_k}_g$, for some $g\in
E(U)$, $k\in\{1,2\}$
\end{prop}
\begin{proof}
By Proposition \ref{prop:characterization of infinite host set} and Propositions \ref{Prop:ExtensionToIsomorphism} and \ref{prop:IsomorphismOpuntoidHost}, $H_e^{S_1*_U S_2}$ is finite if and only if $Host({S\Gamma}(f))$ is
finite. If $Host({S\Gamma}(e))$ is finite, since the automorphism group of a finite tree fixes a vertex or an edge (see \cite[Subsection 27.1.3]{Bab}), then it is straightforward to check that in this case each automorphism $\varphi$ of $Host({S\Gamma}(e))$ has to fix a lobe $\Delta=S\Gamma(X_k,R_k;g)$, for some $k\in\{1,2\}$, $g\in E(U)$. Thus $Aut(Host({S\Gamma}(e)))\simeq Aut (\Delta)$, whence $H_e^{S_1*_U S_2}\simeq H^{S_k}_g$. The converse is trivial.
\end{proof}
For infinite maximal subgroups we have the following characterization.
\begin{theorem}\label{theo: finiteness conditions}
Let $e\in E(S_1*_U S_2)$ with $e \mathcal{D}^{S_1*_US_2}f$ for
some $f\in E(U)$. Then the following are equipvalent:
\begin{enumerate}
\item $H_e^{S_1*_U S_2}$ is infinite;
\item there is a sequence $f_1,f_2,....,f_{2t-2}$ of idempotents of $U$ for some $t>1$ such that:
\begin{itemize}
\item $f \mathcal{D}^{S_1*_US_2}f_{1}$,
\item for each $1<i\le 2t-2$, $f_{i-1}$ and $f_i$ are not $\mathcal{D}^{U}$-related,
\item there is some $k\in \{1,2\}$ such that $f_{1}\mathcal{D}^{S_k}f_{t}\mathcal{D}^{S_k}f_{2t-2}$, and for each $1<i<2t-2$ even $f_{i-1}\mathcal{D}^{S_{3-k}}f_i\mathcal{D}^{S_{k}}f_{i+1}$, and $f_{2t-3}\mathcal{D}^{S_{3-k}}f_{2t-2}$
\end{itemize}
\item $Y=H_e^{S_1*_U S_2}\backslash
\mathcal{T}_e$ is not a tree.
\end{enumerate}
\end{theorem}
\begin{proof}
Again, by Proposition \ref{prop:characterization of infinite host set} and Propositions \ref{Prop:ExtensionToIsomorphism} and\ref{prop:IsomorphismOpuntoidHost}, $H_e^{S_1*_U S_2}$ is infinite if and only if $Host({S\Gamma}(f))$ is
infinite. Moreover by Corollary \ref{cor:characterization of infinite host set} $Host({S\Gamma}(f))$ is infinite if and only if there is a reduced lobe path
$$
P:\Delta_{1},\ldots,\Delta_{t},\ldots,\Delta_{2t-1}
$$
with $\Delta_{1}\simeq\Delta_{t}\simeq\Delta_{2t-1}$. We prove the equivalence $1 \Leftrightarrow 2$ by showing that this geometric characterization is equivalent to the algebraic conditions described in the statement 2.
\\
$1)\Rightarrow 2$) Take any intersection vertex $\nu_i$ of $V(\Delta_{i})\cap V(\Delta_{i+1})$ for $1\le i\le 2t-2$ of $P$. Assume without loss of generality that the color of $\Delta_1$ is $1$, by Proposition \ref{prop:hosts which are sch} we have a sequence 
$$
e_1(\nu_1)=e_2(\nu_1),e_2(\nu_2)=e_1(\nu_2),\ldots, e_2(\nu_{2t-2})=e_1(\nu_{2t-2})
$$ 
of idempotents of $U$ with $e_{1}(\nu_{1})\mathcal{D}^{S_{1}*_{U}S_{2}} f$. Put $f_{i}=e_{1}(\nu_{i})$. Since $\Delta_{1}\simeq\Delta_{t}\simeq\Delta_{2t-1}$, then $f_1\mathcal{D}^{1}f_{t}\mathcal{D}^{1} f_{2t-2}$. Moreover it is straightforward to check that this sequence satisfies also the other conditions of statement 2. 
\\
$2)\Rightarrow 1$) Assuming without loss of generality $k=1$, then
$\Delta_1=S\Gamma(X_1,R_1;f_{1})$ is a host of $\Gamma=S\Gamma(X,R\cup
W;f)\simeq S\Gamma(X,R\cup W; e)$. Let $\nu_1\in V(\Delta_1)$
such that $e_1(\nu_1)=f_1$. Since $\Gamma$ is complete $\nu_1$ is an intersection vertex, so let $\Delta_2$ be the lobe of $\Gamma$ that shares the vertex $\nu_1$ with $\Delta_1$. Then $\Delta_2 =S\Gamma(X_2,R_2;f_1)$ is a host by Proposition \ref{prop:hosts which are sch}. Since $f_{1}\mathcal{D}^{S_2}f_2$ and $f_{2}$ is not $\mathcal{D}^{U}$ related to $f_{1}$, then there is a vertex $\nu_{2}\in V(\Delta_{2})$ which is not connected to $\nu_{1}$ by any path labeled by an element in $U$ and $e_{2}(\nu_{2})=f_{2}$. Thus $\nu_{2}$ does not belong to the intersection vertices of $\Delta_{1}, \Delta_{2}$, and so there is a lobe $\Delta_{3}$, different from $\Delta_{1}$, such that $\nu_{2}\in V(\Delta_{2})\cap V(\Delta_{3})$ and $\Delta_{3}\simeq S\Gamma(X_{1},R_{1};f_{2})$ is a host by Proposition \ref{prop:hosts which are sch}. Using now $f_{2}\mathcal{D}^{S_{1}} f_{3}$ and the fact that $f_{2}$ and $f_3$ are not $\mathcal{D}^{U}$-related we get that there is a vertex $\nu_{3}$ with $e_{1}(\nu_{3})=f_{3}$ which is not an intersection vertex between $\Delta_{2},\Delta_{3}$. Continuing in this way we build a reduced lobe path $P:\Delta_{1},\ldots,\Delta_{t},\ldots,\Delta_{2t-1}
$ such that $\nu_{i}\in V(\Delta_{1})\cap V(\Delta_{i+1})$ for $1\le i\le 2t-2$, with $e_{1}(\nu_{1})=f_{1}, e_{1}(\nu_{t})=f_{t}, e_{1}(\nu_{2t-2})=f_{2t-2}$. Hence, since $f_{1}\mathcal{D}^{S_1}f_{t}\mathcal{D}^{S_1}f_{2t-2}$, we get $\Delta_{1}\simeq\Delta_{t}\simeq\Delta_{2t-1}$.
\\
$1)\Rightarrow 3$) By Proposition \ref{prop:characterization of infinite host set} there is shift-isomorphism $\varphi$. Hence, there is a reduce lobe path $P:\Delta_{1},\ldots,\Delta_{2t-1}$ such that $\Delta_{1}$ is sent onto $\Delta_{2t-1}$ by $\varphi$ and the automorphism on the lobe graph induced by $\varphi$ does not map the edge $(\Delta_{1},\Delta_{2})$ into the edge $(\Delta_{2t-1},\Delta_{2t-2})$. Therefore, these two edges do not belong to the same $H_e^{S_1*_U S_2}$-orbit, so in $Y$ there is a non-trivial loop $P'$ induced by $P$.
\\
$3)\Rightarrow 1$) A reduced loop $P'$ in $Y$ lifts to a reduced lobe path $P:\Delta_{1},$ $\ldots,\Delta_{2t-1}$ in $Host(S\Gamma(f))$ for some $t>1$ and with $\Delta_{1},\Delta_{2t-1}$ belonging to the same $H_e^{S_1*_U S_2}$-orbit. Hence there is an automorphism $\varphi\in Aut(Host(S\Gamma(f)))$ which sends $\Delta_{1}$ onto $\Delta_{2t-1}$. Furthermore, any automorphism does not send the edge $(\Delta_{1},\Delta_{2})$ into the edge $(\Delta_{2t-1},\Delta_{2t-2})$, otherwise $P'$ would no be reduced. Hence, $\varphi|_{\Delta_{1}}:\Delta_{1}\rf\Delta_{2t-1}$ is a shift-isomorphism. Therefore, by Proposition \ref{prop:characterization of infinite host set} $Host(S\Gamma(f))$ is infinite
\end{proof}
From the above theorem we obtain that $Y$ is a tree if and only if $H_{e}^{S_{1}*_{U}S_{2}}$ is finite. This is equivalent to the fact that the only case when $H_{e}^{S_{1}*_{U}S_{2}}$ is isomorphic to iterated amalgams of groups is when $H_{e}^{S_{1}*_{U}S_{2}}$ is finite.
\begin{rem}\label{re:respects the J-order}
We recall that an amalgam $[S_1,S_2;U]$ respects the
$\mathcal{J}$-order if for each $e_1,e_2\in E(U)$,
$e_1\mathcal{J}^{S_k} e_2$ implies $e_1\mathcal{J}^{S_{3-k}} e_2$
for each $k\in\{1,2\}$. With such condition, if $e$ is
$\mathcal{J}^{S_1*_US_2}$-related to some idempotent $f\in U$,
then, using an argument similar to the one
in proof of Theorem \ref{theo: finiteness conditions}, it is straightforward to check that each host is isomorphic to
either $S\Gamma(X_1,R_1;f)$ or $S\Gamma(X_2,R_2;f)$. Therefore
$H_e^{S_1*_U S_2}$ is finite, $|Y|=2$, and
$$
H_e^{S_1*_U S_2}\simeq H_{f}^{S_{1}}*_{H_{f}^{U}} H_{f}^{S_{2}}.
$$
\end{rem}

\subsection{Case 2}

Let $e$ be not $\mathcal{D}^{S_1*_US_2}$-related to any idempotent of $S_1$ or $S_2$. Then $S\Gamma(X,R \cup W;e)$ has a unique host that is a
subopuntoid subgraph of the underlying graph of $Core(f)$.
Thus $H^{S_1*_US_2}_e$ stabilizes some lobes
$\Delta$ of the host. Since this lobe is finite for any $\nu\in
V(\Delta)$ there is a minimum idempotent, namely $e=e_k(\nu)$,
labelling a loop based at $\nu$. Thus, by \cite[Lemma 2]{Finite}
$(\nu,\Delta,\nu)$ is a DV-quotient of the \sch automaton
$\mathcal{A}(X_k,R_k;e)=(\alpha,\Sigma,\alpha)$ called in
\cite{Finite} the maximum determinizing \sch automaton of
$(\nu,\Delta,\nu)$. Denoting by $\pi:(\alpha,\Sigma,\alpha)\rf
(\nu,\Delta,\nu)$ the natural homomorphism induced by this
quotient we show that we can lift an automorphism $\phi$ of $\Delta$ to
an automorphism $\varphi$ of $\Sigma$ for which the following diagram
commutes:
$$
\setlength{\unitlength}{0.5mm}
\begin{picture}(63,48)(-32,-24)
\put(-32,-24){}
\node[Nframe=n,NLangle=0.0](n0)(-23.96,17.76){$\Sigma$}

\node[Nframe=n,NLangle=0.0](n1)(24.05,17.76){$\Sigma$}

\node[Nframe=n,NLangle=0.0](n2)(-23.96,-18.24){$\Delta$}

\node[Nframe=n,NLangle=0.0](n3)(24.05,-18.24){$\Delta$}

\drawedge(n0,n1){$\varphi$}

\drawedge(n2,n3){$\phi$}

\drawedge[ELside=r](n0,n2){$\pi$}

\drawedge(n1,n3){$\pi$}

\drawcbezier(3.14,7.28,-8.93,11.93,-14.64,-10.08,7.58,-5.42)
\end{picture}
$$

\begin{theorem}\label{lifting property}
Let $(\nu,\Delta,\nu)$ be a closed inverse automaton
relative to the presentation $\langle X_k|R_k\rangle$ for some
$k\in\{1,2\}$. With the above notation let
$(\alpha,\Sigma,\alpha)$ be the maximum determinizing \sch
automaton of $(\nu,\Delta,\nu)$ where $\pi(\alpha)=\nu$. Then every
automorphism $\phi\in Aut(\Delta)$ can be lifted to an
automorphism $\varphi\in Aut(\Sigma)$ such that $\varphi\circ\pi=\pi\circ\phi$. Moreover there is a group epimorphism from the subgroup $H:=\{\varphi\in
Aut(\Sigma):\exists\phi\in Aut(\Delta),
\varphi\circ\pi=\pi\circ\phi\}$ of $Aut(\Sigma)$ onto
$Aut(\Delta)$ with kernel $N=H\cap S$ where $S=\{\varphi\in
Aut(\Sigma):\varphi(\pi^{-1}(\nu))\subseteq \pi^{-1}(\nu)\}$.
\begin{proof}
Let $\phi\in Aut(\Delta)$, let $\nu'=\phi(\nu)$. Since $\phi$ is
labelling preserving, then
$e_k(\nu')=e_k(\nu)=e$. Thus there is a word $w\in (X\cup
X^{-1})^*$ labelling a path $(\nu,w,\nu')$ in $\Delta$ such that
$ww^{-1}=w^{-1}w=e$. Since
$(\alpha,\Sigma,\alpha)=\mathcal{A}(X_k,R_k;e)$ there is also a
path $(\alpha,w,\alpha')$ for some $\alpha'\in V(\Sigma)$ and by
the minimality of $e_k(\nu')$ we get $e_k(\alpha')=e$. Therefore
$(\alpha',\Sigma,\alpha')$ and $(\alpha,\Sigma,\alpha)$ are \sch
automata that accept the same language, hence by Proposition
\ref{prop:old Ste} there is an automorphism $\varphi\in
Aut(\Sigma)$ such that $\varphi(\alpha)=\alpha'$. We prove that
$\varphi$ is the automorphism satisfying the lifting property
$\varphi\circ\pi=\pi\circ\phi$. For this purpose let $v$ be a
vertex of $\Sigma$ and let $r\in (X\cup X^{-1})^+$ be a word
labelling a path $(\alpha,r,v)$, so applying the automorphism
$\varphi$ this path goes to $(\alpha',r,v')$ with $v'=\varphi(v)$.
Consider $\pi(v)$ then clearly $(\nu,r,\pi(v))$ is a path in
$\Delta$ thus the image of this path by $\phi$ is
$(\nu',r,\phi(\pi(v)))$, hence $(\nu,wr,\phi(\pi(v)))$ is also a
path in $\Delta$. Consider now $\pi(v')$, since $(\alpha, wr,v')$
is a path in $\Sigma$ then $(\nu,wr,\pi(v'))$ is also a path in
$\Delta$, whence by the determinism of $\Delta$ we get
$\phi(\pi(v))=\pi(v')=\pi(\varphi(v))$.
\\
Let $H:=\{\varphi\in Aut(\Sigma):\exists\phi\in Aut(\Delta),
\varphi\circ\pi=\pi\circ\phi\}$.
It is straightforward checking that $H$ is a subgroup of $Aut(\Delta)$.

For any $\varphi\in H$, the relation
$\pi^{-1}\circ\varphi\circ\pi\subseteq V(\Delta)\times V(\Delta)$
is a function, since by definition of $H$ there is a $\phi$ such
that $\varphi\circ\pi=\pi\circ\phi$ and so, taking into account
that $\pi$ is surjective, then for any left inverse $\pi^{-1}$ we have
$$
\pi^{-1}\circ(\varphi\circ\pi)=\pi^{-1}\circ(\pi\circ\phi)=
(\pi^{-1}\circ\pi)\circ\phi=1_{\Delta}\circ\phi=\phi
$$
So there is a map $\lambda: H\rf Aut(\Delta)$ defined by
$\lambda(\varphi)=\pi^{-1}\circ\varphi\circ\pi$. Moreover
$\lambda$ is surjective since by the first statement of the
theorem for any $\phi\in Aut(\Delta)$ there is a $\varphi\in
Aut(\Sigma)$ such that $\pi\circ\phi=\varphi\circ\pi$ and so
$\varphi\in H$ and $\lambda(\varphi)=\phi$. It remains to show
that $\lambda$ is a homomorphism. Let $\varphi_i\in H$ and let
$\phi_i\in Aut(\Delta)$ such that
$\varphi_i\circ\pi=\pi\circ\phi_i$ for $i=1,2$, by the definitions we get 
\begin{eqnarray}
\nonumber  \lambda(\varphi_1\circ\varphi_2)&=&\pi^{-1}\circ(\varphi_1\circ\varphi_2)\circ\pi=\\
\nonumber&=&(\pi^{-1}\circ\varphi_1)\circ(\pi\circ\phi_2)=(\pi^{-1}\circ\varphi_1\circ\pi)\circ\phi_2=\\
\nonumber&=&\lambda(\varphi_1)\circ\phi_2=\lambda(\varphi_1)\circ\lambda(\varphi_2)
\end{eqnarray}
The last statement is a routine calculus which involves only the definitions of $H$ and $S$.
\end{proof}
\end{theorem}

Note that without the finiteness condition of the inverse
semigroup $S$, in general it is not possible to define the maximum
determinizing \sch automaton of a closed inverse word automaton
relative to the presentation $\langle X|R\rangle$ of the inverse
semigroup $S$. It is also quite easy to produce an example where
it is not possible to lift an automorphism of a closed DV-quotient
$\Delta$ of a \sch automaton $\Sigma$ to an automorphism of
$\Sigma$ (see \cite{Rthesys}). Moreover the subgroup $H$ in the
previous theorem is in general a proper subgroup of $Aut(\Sigma)$.
To prove this fact it is enough to consider the dihedral group
$D_4=Gp\langle r,s|r^2,s^2,(rs)^4 \rangle$ which is clearly a
finite inverse semigroup with only one \sch graph which is the
Cayley graph of $D_4$. If in the Cayley graph of $D_4$ we identify
the identity $e$ with $s$ and then we determinize, we obtain an
inverse word graph $\Delta$ with $Aut(\Delta)\simeq Gp\langle
\sigma|\sigma^2\rangle$. It is easy to show that $\sigma$ can be
lifted up to the automorphism $(sr)^2$, however the automorphism
$(sr)$ is not in $H$.\\

The next theorem covers the last case.
\begin{theorem}
  Let $e\in E(S_1*_U S_2)$ with $S_1,S_2$ finite inverse semigroups
  and suppose that $e$ is not $\mathcal{D}^{S_1*_US_2}$-related to any idempotent of $S_1$ or $S_2$. Therefore $H_e^{S_1*_US_2}$ is a homomorphic image
  of some subgroup of the maximal subgroup $H^{S_k}_g$ of $S_k$ for some $k\in\{1,2\}$ and $g\in E(S_k)$.
\end{theorem}
\begin{proof}
  We already remarked that in this case $H_e^{S_1*_US_2}$ is isomorphic to the automorphism group of
   some lobe $\Delta$ of $Host(S\Gamma(e))$. By Theorem \ref{lifting property} $Aut(\Delta)$ is an homomorphic
  image of $Aut(\Sigma)$ where $\Sigma=S\Gamma(X_k,R_k;g)$ for
  some $g\in E(S_k)$. Therefore $H_e^{S_1*_US_2}$ is a homomorphic image of
  $Aut(\Sigma)\simeq H_g^{S_k}$. In particular $H_e^{S_1*_US_2}\simeq H_g^{S_k}/N$ where
  $N$ is the normal subgroup described in Theorem \ref{lifting property}.
\end{proof}

\begin{rem}
We note that when $S_1$ and $S_2$ are $E$-unitary then no quotient
has to be performed in the construction of the \sch graph of some
word with respect to the standard presentation of the amalgam.
Then, if $e$ is not $\mathcal{D}^{S_1*_US_2}$-related to any idempotent of $S_1$ or $S_2$, $H_e^{S_1*_US_2}$ is isomorphic to the
maximal subgroup $H^{S_k}_g$ of $S_k$ for some $k\in\{1,2\}$ and
$g\in E(S_k)$.
\end{rem}

\section{Conclusion}

We have completely determined the  structure of the maximal subgroups of the amalgamated free-product of an amalgam of finite inverse semigroups.
All these groups are finitely presented, and we sketch the proof that their presentations are effectively computable. For more details of the proof see the final chapter of \cite{Rthesys}.
\begin{theorem}
  Let $e\in E(S_1*_U S_2)$ with $S_1,S_2$ finite inverse semigroups,
  then there is an algorithm to compute a presentation of $H_e^{S_{1}*_{U}S_{2}}$.
\end{theorem}
\begin{proof}
 If the host is unique then $H_e^{S_1*_US_2}$ is the automorphism group of a lobe $\Delta$ of the host. The host is finite, so such a lobe can be determined as well as a presentation of $Aut(\Delta)$. If $S\Gamma(e)$ has more than one host, it is enough to find a maximal subtree of $Y$ and then to compute the presentation of $H_e^{S_{1}*_{U}S_{2}}$. Starting from any host in $Core(u)$ for some word $u$ equivalent to $e$ in $S_{1}*_{U}S_{2}$, we can build a maximal subtree of $Y$ recursively adding at each step adjacent hosts which are non-isomorphic
 to the ones previously chosen. It is straightforward to check that when we obtain an opuntoid graph $\Theta$ for which each adjacent host is isomorphic to a lobe occuring in $\Theta$, then by Propositions \ref{Prop:ExtensionToIsomorphism} and \ref{prop:IsomorphismOpuntoidHost}, all the lobes of $\Theta$ are representatives of all the orbits. 
\end{proof}
We end the section considering the case when $S_{1}, S_{2}$ are combinatorial. Note that a finite inverse semigroup which is combinatorial is a semilattice. Thus, in our case any \sch automaton is formed by at most two adjacent lobes. Hence, as an easy consequence of the above results, we have the following
\begin{cor}
 Let ${S_1,S_2;U}$ be an amalgam of finite inverse semigroups, then $S_1*_US_2$ is combinatorial if and only if $S_1$ and $S_2$ are both combinatorial.
\end{cor}

\section*{Acknowledgments}

The last author acknowledges the support from the European Regional Development Fund through the
programme COMPETE and by the Portuguese Government through the FCT -- Funda\c c\~ao para a Ci\^encia e a Tecnologia under the project
PEst-C/MAT/UI0144/2011 as well as the support of the FCT project SFRH/BPD/65428/2009.

\bibliography{biblio}
\bibliographystyle{plain}

\end{document}